\documentclass[11pt]{amsart}
\usepackage{geometry}                
\usepackage{enumerate, amsmath, amsfonts, amssymb, amsthm, wasysym, graphics, graphicx, xcolor, colortbl, url, hyperref, hypcap, ifthen, xargs, pdflscape, caption, tikz, rotate, multicol, multirow}
\usepackage[font={sl,footnotesize}]{subcaption} 

\hypersetup{colorlinks=true, citecolor=darkblue, linkcolor=darkblue}

\definecolor{darkblue}{rgb}{0,0,0.7} 
\newcommand{\darkblue}{\color{darkblue}} 
\newcommand{\defn}[1]{\emph{\darkblue #1}} 

\newcommand{\blue}[1]{{\color{blue} #1}} 
\newcommand{\red}[1]{\underline{\bf \color{red} #1}} 

\newtheorem{theorem}{Theorem}[section]
\newtheorem{proposition}[theorem]{Proposition}
\newtheorem{corollary}[theorem]{Corollary}
\newtheorem{lemma}[theorem]{Lemma}

\theoremstyle{definition}
\newtheorem{definition}[theorem]{Definition}
\newtheorem{example}[theorem]{Example}

\newtheorem{remark}[theorem]{Remark}

\newcommand{\ssm}{\smallsetminus} 

\newcommand{\YY}{{\mathcal Y}}

\newcommand{\A}{{\mathcal A}}
\newcommand{\C}{{\mathcal C}}
\newcommand{\R}{{\mathbb R}}

\newcommand{\subwordComplex}[2]{\mathcal{SC}(#1,#2)} 
\newcommand{\wordprod}[1]{\sigma_{#1}} 
\newcommand{\Root}[2]{\mathsf{r}(#1,#2)} 
\newcommand{\RootF}[2]{\mathsf{r}_F(#1,#2)} 
\newcommand{\Roots}[1]{\mathsf{R}(#1)} 
\newcommand{\wo}{w_\circ} 
\newcommand{\shifted}[1]{{#1}}
\newcommand{\rotated}[1]{#1^{\circlearrowleft}}
\newcommand{\woconj}[1]{#1^{\wo}}

\newcommand{\seed}{A} 
\newcommand{\cluster}{T} 

\definecolor{Gray}{gray}{0.9} 
\newcolumntype{g}{>{\columncolor{Gray}}c}
\newcolumntype{x}{@{\hspace{.1cm}}c@{\hspace{.1cm}}}

\def\dim{\operatorname{dim}}
\def\inv{\operatorname{inv}}

\newcommand*\circled[1]{\tikz[baseline=(char.base)]{
  \node[shape=circle,draw,inner sep=1pt] (char) {$#1$};}}

\usepackage{todonotes}

\graphicspath{{figures/}}

\title{A Hopf algebra of subword complexes}

\author[Bergeron, Ceballos]{Nantel Bergeron$^{1,2}$, Cesar Ceballos$^{1,2}$}

\address[1]{Fields Institute\\ Toronto, ON, Canada}
\address[2]{York University\\ Toronto, ON, Canada}

\email{bergeron@yorku.ca}
\email{ceballos@mathstat.yorku.ca}

\thanks{
The first author was partially  supported by NSERC.\\
The second author was supported by the government of Canada through a Banting Postdoctoral Fellowship. He was also supported by a York University research grant.
}

\date{\today}

\begin{document}
\maketitle

\begin{abstract}
We introduce a Hopf algebra structure of subword complexes, including both finite and infinite types. 
We present an explicit cancellation free formula for the antipode using acyclic orientations of certain graphs, and show that this Hopf algebra induces a natural non-trivial sub-Hopf algebra on $c$-clusters in the theory of cluster algebras. 
%
\end{abstract}

%
%
%
%

\section{Introduction}
\bigskip

Subword complexes are simplicial complexes introduced by Knutson and Miller, and are motivated by the study of Gr\"obner geometry of Schubert varieties~\cite{KnutsonMiller-subwordComplex,knutson_grobner_2005}.  
These complexes have been shown to have striking connections with diverse areas such as associahedra~\cite{Stasheff, Stasheff-operads, TamariFestschrift}, multi-associahedra~\cite{jonsson_generalized_2005,soll_type-b_2009}, pseudotriangulation polytopes~\cite{RoteSantosStreinu-polytopePseudotriangulations,RoteSantosStreinu-survey}, and cluster algebras~\cite{FominZelevinsky-ClusterAlgebrasI,FominZelevinsky-ClusterAlgebrasII}.

The first connection between subword complexes and associahedra was discovered by Pilaud and Pocchiola who showed that every multi-associahedron can be obtained as a well chosen type~$A$ subword complex in the context of sorting networks~\cite{PilaudPocchiola}. A particular instance of their result was rediscovered using the subword complex terminology in~\cite{Stump,serrano_maximal_2012}. These results were generalized to arbitrary finite Coxeter groups by Ceballos, Labb\'e and Stump in~\cite{ceballos_subword_2013}.  
The results in~\cite{ceballos_subword_2013} provide an additional connection with the $c$-cluster complexes in the theory of cluster algebras, which has been used as a keystone for decisive results about denominator vectors in cluster algebras of finite type~\cite{CeballosPilaud}.
A construction of certain brick polytopes of spherical subword complexes is presented in~\cite{PilaudSantos-brickPolytope,PilaudStump-brickPolytopes}, which is used to give a precise description of the toric varieties of $c$-generalized associahedra in connection with Bott-Samelson varieties in~\cite{escobar_brick_2014}.
More recent developments on geometric and combinatorial properties of subword complexes are presented in~\cite{ceballos_fan_2014,escobar_subword_2015,stump_cataland_2015}.

This paper presents a more algebraic approach to subword complexes. We introduce a Hopf algebra structure on the vector space generated by all facets of irreducible subword complexes, including both finite and infinite types. Such facets include combinatorial objects such as triangulations and multi-triangulations of convex polygons, pseudotriangulations of any planar point set in general position, and $c$-clusters in cluster algebras of finite type.  
We present an explicit cancellation free formula for the antipode using acyclic orientations of certain graphs. It is striking to observe that we to obtain a result very similar to the antipode formula of Humpert and Martin for the incidence Hopf algebra of graphs~\cite{HumpertMartin}. As in~\cite{BenedettiSagan},  our combinatorial Hopf algebra
is part of a nice family  with explicit cancelation free formula for the antipode.
The Hopf algebra of subword complexes also induces a natural sub-Hopf algebra on $c$-clusters of finite type.  
Cluster complexes for Weyl groups were introduced by Fomin and Zelevinsky in connection with their proof of Zamolodchikov's periodicity conjecture for algebraic $Y$-systems in~\cite{FominZelevinsky-YSystems}. These complexes encode the combinatorial structure behind the associated cluster algebra of finite type~\cite{FominZelevinsky-ClusterAlgebrasII}, and are further extended to arbitrary Coxeter groups by Reading in~\cite{Reading-coxeterSortable}. The resulting~$c$-cluster complexes use a Coxeter element~$c$ as a parameter and have been extensibly used to produce geometric constructions of generalized associahedra~\cite{ReadingSpeyer,HohlwegLangeThomas,stella_polyhedral_2011,PilaudStump-brickPolytopes}.  
The basis elements of our Hopf algebra of $c$-clusters are given by pairs of clusters $(\seed,\cluster)$ of finite type,
where $\seed$ is any acyclic cluster seed and $\cluster$ is any cluster obtained from it by mutations. The multiplication and comultiplication operations are natural from the cluster algebra perspective on $\cluster$. However, subword complexes allow us to nontrivially extend these operations to remarkable operations on the acyclic seed  $\seed$.  

The initial motivation of this paper was to extend the Loday-Ronco Hopf algebra on planar binary trees~\cite{loday_hopf_1998} in the context of subword complexes, and to present an algebraic approach to subword complexes that helps to better understand their geometry. Although we can explicitly describe the Loday-Ronco Hopf algebra from the subword complex approach, the Hopf algebra described in this paper differs from our original intent for several reasons: it allows an extension to arbitrary Coxeter groups, it restricts well to the context of $c$-clusters, and contains more geometric information about subword complexes. Our description of the Loday-Ronco Hopf algebra in terms of certain subword complexes of type $A$ will be presented in a forthcoming paper in joint work with Pilaud~\cite{ceballosp1}.  The geometric intuition behind the Hopf algebra of subword complexes presented in this paper was indirectly used to produce the geometric realizations of type~$A$ subword complexes and multi-associahedra of rank~3 in~\cite{ceballos_fan_2014}. 

The outline of the paper is as follows. In Section~\ref{sec:subwordcomplexes} we present the concept of subword complexes, some examples and a decomposition theorem needed for the Hopf algebra structure. In Section~\ref{sec:Hopf} we give the Hopf structure, and compute explicitly a cancelation free formula for the antipode in Section~\ref{sec:antipode}. In Section~\ref{sec:clusters} we show that this Hopf algebra induces a sub-Hopf algebra on $c$-clusters of finite type and present a combinatorial model description for Cartesian products of classical types. We also have two small appendices. In Appendix~\ref{Appendix:inversions}, we geometrically study  the sequence of inversions of a word (not necessarily reduced) in the generators of a Coxeter group. This will be useful for our decomposition theorem of subword complexes in Section~\ref{sec:subwordcomplexes}. In Appendix~\ref{Appendix:gems} we give an interpretation of the top-to-random shuffle operator on our Hopf algebra. This gives an example of a rock breaking process as in \cite{DiaconisPangRam, Pang2014} that may have more than one different stable outcome.

\medskip
\noindent
{\bf Acknowledgements:}
The proof in Section~\ref{sec:cancel free} is based on discussions with Carolina Benedetti and Bruce Sagan.  The involution we introduce is very close to the one presented in~\cite{BenedettiSagan}.
We are especially grateful to Nathan Reading for his help with the proof of Lemma~\ref{lem:parabolic_sorting}, and to Christophe Hohlweg for his help with the generalization of our Hopf algebra to infinite Coxeter groups. We are grateful to Vincent Pilaud, Salvatore Stella and Jean-Philippe Labb\'e for helpful discussions. We also thank the Banting Postdoctoral Fellowships program of the government of Canada and York University for their support on this project.

%
%
%
%

\section{Subword Complexes}\label{sec:subwordcomplexes}
Let~$W$ be a possibly infinite Coxeter group with generators~$S=\{s_1,\dots , s_n\}$. This group acts on a real vector space~$V$, we denote by~$\Delta:=\{\alpha_s \ |\  s\in S\}$ the set of simple roots of a root system~$\Phi=\Phi^+ \sqcup  \Phi^-$ associated to~$W$. Throughout the paper, for simplicity,  we think of $W$ as the tuple $(W,S,\Phi^+)$ containing the information of the group, its generators and the decomposition of its root system $\Phi=\Phi^+ \sqcup  -(\Phi^+)$.

\begin{definition}[\cite{KnutsonMiller-subwordComplex}]
Let~$Q=(q_1,\dots,q_r)$ be a word in $S$ and $\pi\in W$ be an element of the group. The \defn{subword complex}~$\subwordComplex{Q}{\pi}$ is a simplicial complex whose faces are given by subsets $I\subset [r]=\{1,2,\ldots,r\}$, such that the subword of~$Q$ with positions at~$[r]\ssm I$ contains a reduced expression of~$\pi$.  
\end{definition}

\begin{example}
\label{ex:A2}
%
%
%
%
%
%
%

Let $W=\mathcal S_3$ be the symmetric group generated by the simple transpositions $S=\{s_1,s_2\}=\{(1\ 2), (2\ 3)\}$. Let $Q=(q_1,q_2,q_3,q_4,q_5)=(s_1,s_2,s_1,s_2,s_1)$ and $\pi=s_1s_2$. 
Since the reduced expressions of $\pi$ in $Q$ are given by $q_1q_2=q_1q_4=q_3q_4=\pi$, 
the maximal faces of~$\subwordComplex{Q}{\pi}$ are $\{3,4,5\},\{2,3,5\}$ and $\{1,2,5\}$. This subword complex is illustrated in Figure~\ref{fig:A2}, where we use the network diagrams used by Pilaud and Pocchiola in~\cite{PilaudPocchiola}. Such diagrams will be used through out the paper to represent subword complexes of type $A$. The letters in the word $Q$ are consecutively placed form left to right as vertical commutators in the diagram such that a generator $s_i$ connects the horizontal levels $i$ and $i+1$ numerated from bottom to top. Figure~\ref{fig:A2} also illustrates the three possible facets in the network diagram.  

\begin{figure}[htbp]
\begin{center}
\includegraphics[width=0.9\textwidth]{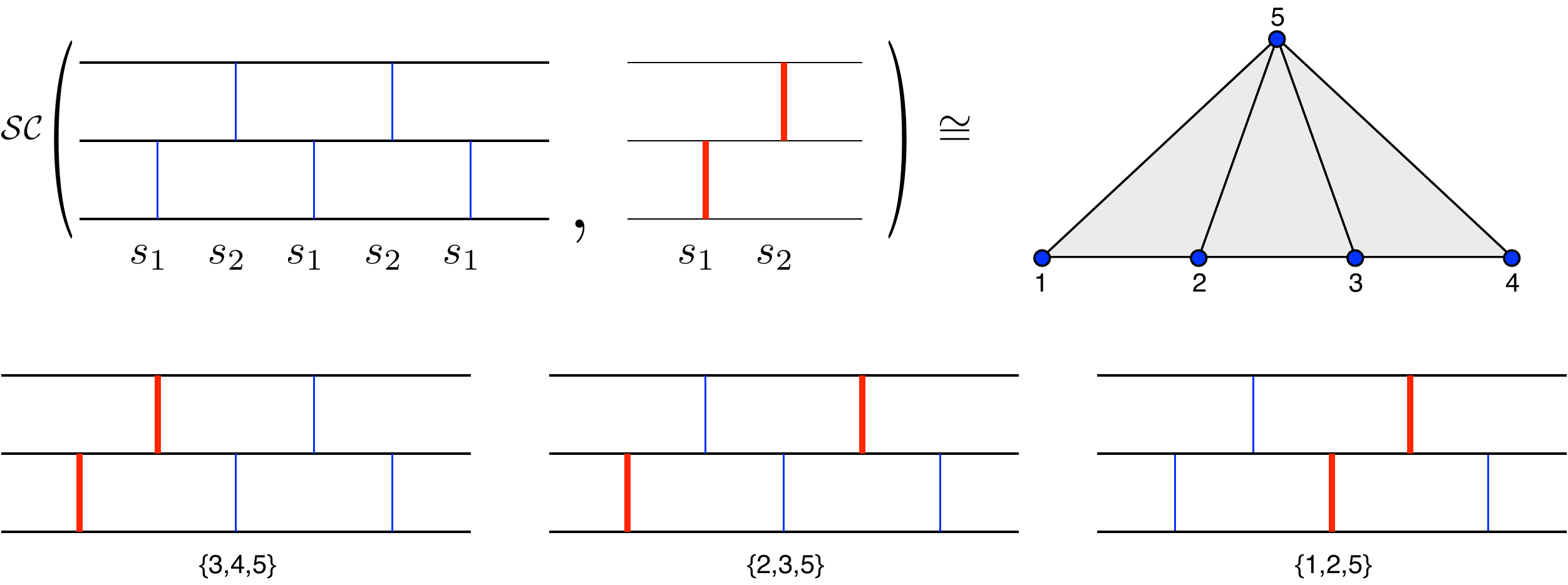}
\caption{Subword complex~$\subwordComplex{Q}{\pi}$ for $Q=(s_1,s_2,s_1,s_2,s_1)$ and $\pi=s_1s_2$ in type $A_2$ (top). Its maximal faces are $\{3,4,5\},\{2,3,5\}$ and $\{1,2,5\}$ (bottom).}
\label{fig:A2}
\end{center}
\end{figure}

\end{example}

Two remarkable examples of subword complexes are the dual associahedron and the multi-associahedron. The first description of these two complexes as well chosen subword complexes was given by Pilaud and Pocchiola in the context of sorting networks in~\cite[Section~3.3 and Theorem~23]{PilaudPocchiola}. A particular case of their result was rediscovered by Stump~\cite{Stump} and Stump and Serrano~\cite{serrano_maximal_2012}, who explicitly used the terminology of subword complexes in type $A$. We refer to~\cite[Section~2.4]{ceballos_subword_2013} for a precise description of these two complexes in the generality of~\cite{PilaudPocchiola} and a generalization of their results to arbitrary finite Coxeter groups. 

\begin{example}[Associahedron]

Let $W=\mathcal S_4$ be the symmetric group generated by the simple transpositions $S=\{s_1,s_2,s_3\}=\{(1\ 2), (2\ 3), (3\ 4)\}$. Let $Q=(s_1,s_2,s_3,s_1,s_2,s_3,s_1,s_2,s_1)$ and $\pi=[4\ 3\ 2\ 1]$ be the longest element of the group. The subword complex~$\subwordComplex{Q}{\pi}$ is isomorphic to the boundary complex of the dual of the 3-dimensional associahedron. The vertices of this complex correspond to diagonals of a convex $6$-gon and the facets to its triangulations.
Figure~\ref{fig:A3k1} illustrates the facet at positions $\{3,4,5\}$ and its corresponding triangulation of the polygon. The bijection sends the $i$th letter in $Q$ to the $i$th diagonal of the polygon in lexicographic order. A set of positions in $Q$ forms a facet of the subword complex if and only if the corresponding diagonals form a triangulation of the polygon. 
Figure~\ref{fig:c-clusters_classical_typeA} illustrates an example of a more general version of this bijection, which is explained in Section~\ref{sec:c-clusters}. 


\begin{figure}[htbp]
\begin{center}
\includegraphics[width=0.95\textwidth]{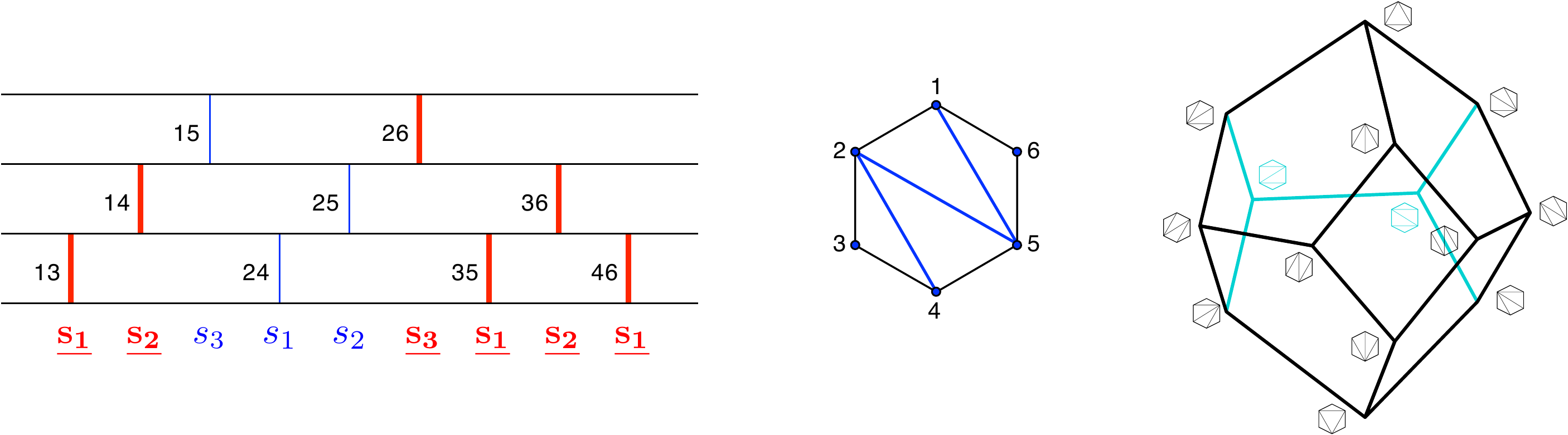}
\caption{Bijection between the facets of the illustrated subword complex and triangulations of a convex $6$-gon. The letters in the word correspond to the diagonals of the polygon ordered in lexicographic order. The subword complex is isomorphic to the dual of the 3-dimensional associahedron illustrated on the right.}
\label{fig:A3k1}
\end{center}
\end{figure}

\end{example}

\begin{example}[Multi-associahedron]
\label{ex:A3k2}

Let $W=\mathcal S_4$ and $S=\{s_1,s_2,s_3\}=\{(1\ 2), (2\ 3), (3\ 4)\}$ as above, $Q=(s_1,s_2,s_3,s_1,s_2,s_3,s_1,s_2,s_3,s_1,s_2,s_1)$ and $\pi=[4\ 3\ 2\ 1]$. 
The subword complex~$\subwordComplex{Q}{\pi}$ is isomorphic to the (simplicial) multi-associahedron $\Delta_{8,2}$. The vertices of this complex are the 2-relevant diagonals of a convex 8-gon, that is diagonals that leave at least two vertices of the polygon on each of its sides. The faces are subsets of 2-relevant diagonals not containing a 3-crossing, that is 3 diagonals that mutually cross in their interiors.
The thick blue diagonals in the right part of Figure~\ref{fig:A3k2} form a maximal set of 2-relevant diagonals not containing a 3-crossing. The corresponding facet~$I=\{1,3,7,8,9,10\}$ of the subword complex is illustrated on the left. The bijection sends the $i$th letter in $Q$ to the $i$th 2-relevant diagonal of the polygon in lexicographic order. Note that the thin grey diagonals in the figure never appear in a 3-crossing and therefore are considered to be ``irrelevant". A maximal set of diagonals (relevant or not) of a polygon not containing a $(k+1)$-crossing is known in the literature as a $k$-triangulation.  
We refer to~\cite[Section~2]{ceballos_subword_2013} for a more details on this bijection in the generality of~\cite{PilaudPocchiola}.

\begin{figure}[htbp]
\begin{center}
\includegraphics[width=0.95\textwidth]{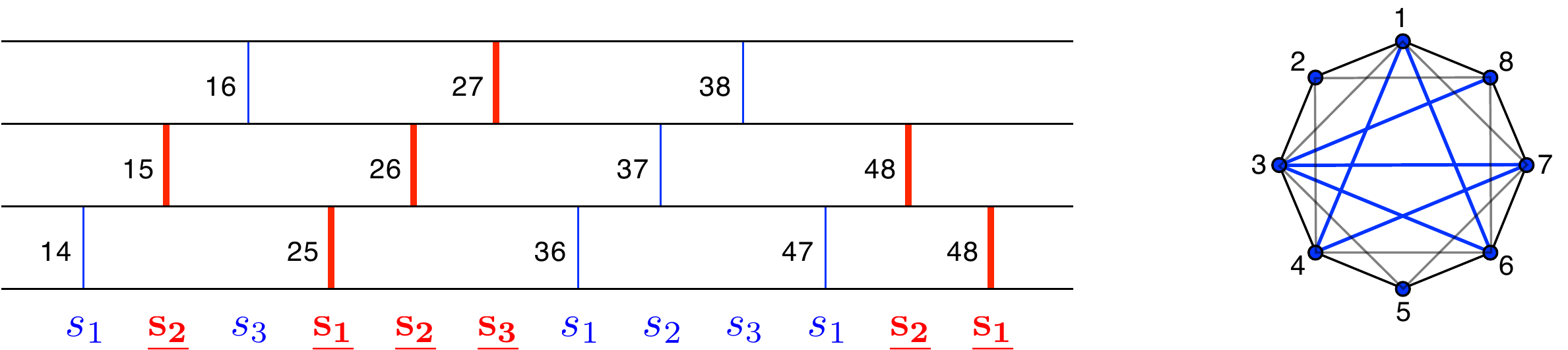}
\caption{Bijection between the facets of the illustrated subword complex and $2$-triangulations of a convex $8$-gon. The letters in the word correspond to the $2$-relevant diagonals of the polygon ordered in lexicographic order. The subword complex is isomorphic to the multi-associahedron whose facets correspond to $2$-triangulations of an $8$-gon. The thin grey diagonals are irrelevant because they appear in every $2$-triangulation.}
\label{fig:A3k2}
\end{center}
\end{figure}

\end{example}

The multi-associahedron is a rich combinatorial object that is conjectured to be realizable as the boundary complex of a convex polytope~\cite[Section 1.2]{jonsson_generalized_2005}. Inspired by our Hopf algebra of subword complexes, we discovered certain geometric constructions of a particular family of multi-associahedra~\cite{ceballos_fan_2014}. 
Another important family of examples in connection with cluster complexes in the theory of cluster algebras, and the corresponding induced Hopf algebra will be presented in Section~\ref{sec:clusters}.

\subsection{Root function and flats}
Associated to a subword complex, one can define a root function which plays a fundamental role in the theory. This function was introduced by Ceballos, Labb\'e and Stump in~\cite{ceballos_subword_2013}. It encodes exchanges in the facets of the subword complex~\cite{ceballos_subword_2013} and has been extensively used in the construction of Coxeter brick polytopes~\cite{PilaudStump-brickPolytopes} and in the description of denominator vectors in cluster algebras of finite type~\cite{CeballosPilaud}.

\begin{definition}[\cite{ceballos_subword_2013}]
\label{def:rootFunction}
The \defn{root function}
\[
\Root{I}{\cdot} : [r] \longrightarrow \Phi
\]
associated to a subset $I \subseteq [r]$ is defined by
\[
\Root{I}{j} := \wordprod{A_j}(\alpha_{q_j}),
\]
where~$A_j := [j-1]\ssm I$ is the set of positions on the left of $j$ that are in the complement of~$I$, and~$\wordprod{X} \in W$ denotes the product of the elements $q_x \in Q$ for $x \in X$ in the order they appear in~$Q$. 
The \defn{root configuration} of~$I$ is the list $\Roots{I} := (\Root{I}{i}:i \in I)$. 
We denote by~$\Root{I}{Q}$ the list of roots~$(\Root{I}{1},\dots , \Root{I}{r})$.
\end{definition}

All the information about the subword complex is encoded by its root function. In particular, the flips between facets can be described as follows. Lemma~\ref{lem:subword_flips} was stated for subword complexes of finite type in~\cite{ceballos_subword_2013}, but the proof works exactly the same for arbitrary Coxeter groups (finite or not).

\begin{lemma}[{\cite[Lemmas 3.3 and 3.6]{ceballos_subword_2013}}]
\label{lem:subword_flips}
Let $I$ and $I'$ be two adjacent facets of the subword complex~$\subwordComplex{Q}{\pi}$ with $I\ssm i = I' \ssm i'$.
\begin{enumerate}[(1)]
\item The position $i'$ is the unique position in the complement of $I$ such that $\Root{I}{i'}\in \{\pm \Root{I}{i}\}$. Moreover, $\Root{I}{i'}=\Root{I}{i}$ if $i<i'$, while $\Root{I}{i'}=-\Root{I}{i}$ if $i'<i$.
\item The map $\Root{I'}{\cdot}$ is obtained from the map $\Root{I}{\cdot}$ by
\[
\Root{I'}{k}=
\begin{cases}
s_{\Root{I}{i}}\Root{I}{k} & \text{if } \operatorname{min} \{i,j\} < k \leq \operatorname{max} \{i,j\} \\
\Root{I}{k} & \text{otherwise}
\end{cases}
\]
where $s_{\Root{I}{i}}\in W$ denotes the reflection that is orthogonal (or dual) to the root $\Root{I}{i}$.
\end{enumerate}

\end{lemma}

\begin{example}[Example~\ref{ex:A3k2} continued]
Let $\{\alpha_1,\dots ,\alpha_n\}$ be the simple roots of the root system of type~$A_n$. The positive roots can be written as positive linear combinations 
$
\alpha_{i\dots j} = \sum_{\ell =i}^j \alpha_\ell,
$
for $1\leq i \leq j \leq n$, and the negative roots are the roots $-\alpha_{i\dots j}$.
The group acts on the roots according to the following rule which is extended by linearity, 
\[
s_i(\alpha_j) =
\begin{cases}
  -\alpha_j &\text{ if } i=j,\\
  \alpha_i+\alpha_j &\text{ if } |i-j|=1,\\
  \alpha_j & \text{otherwise.}
\end{cases}
\]
The root function of the subword complex in Example~\ref{ex:A3k2} with respect to the facet~$I=\{1,3,7,8,9,10\}$ is illustrated in Figure~\ref{fig:A3k2_roots}. It associates a root to each of the letters in the word, the root $\alpha_{i\dots j}$ would be represented in the diagram by the indices $i\dots j$ for simplicity. For example, the indices 23 represent the root $\alpha_{23}=\alpha_2+\alpha_3$. In order to distinguish these indices with the ones used in Figures~\ref{fig:A3k1} and~\ref{fig:A3k2}, indices corresponding to diagonals of a polygon are placed on the left of each commutator, while indices corresponding to roots are placed on the right throughout the paper. The root associated to a letter $s_j$ in~$Q$ can be thought as the underlined red word on the left of that letter applied to $\alpha_j$.

\begin{figure}[htbp]
\begin{center}
\includegraphics[width=0.6\textwidth]{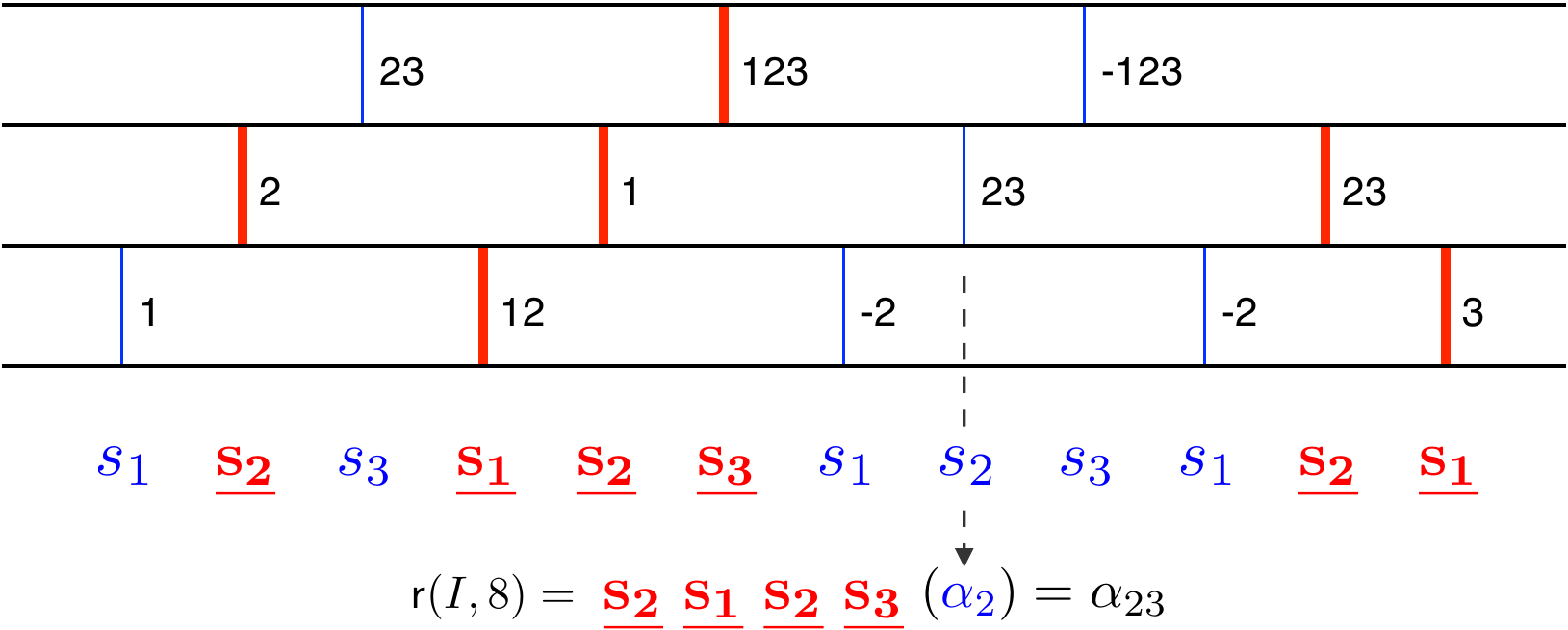}
\caption{Root function of the subword complex in Example~\ref{ex:A3k2} for the facet $I=\{1,3,7,8,9,10\}$.}
\label{fig:A3k2_roots}
\end{center}
\end{figure}

Note that exchanges in facets can be easily performed knowing the root function. For example, any of the two thin blue commutators labeled 23 can be flipped to the unique bold red commutator 23 to form a new facet. In contrast, any 2-relevant blue diagonal in the 2-triangulation in Figure~\ref{fig:A3k2} can be flipped to a unique diagonal to form a new 2-triangulation. However such flips are much easier to visualize in the subword complex. We refer to~\cite{pilaud_multitriangulations_2009} for a description of these flips using star polygons directly in the $k$-triangulations. 

\end{example}

\begin{definition}
A subword complex~$\subwordComplex{Q}{\pi}$ is said to be \defn{irreducible} if and only if the root configuration~$\Roots{I}$ generates the vector space~$V$ for some facet~$I$. Or equivalently, if the root configuration~$\Roots{I}$ generates the vector space~$V$ for any facet~$I$ (these two conditions are equivalent by Lemma~\ref{lem:subword_flips}(2) and the fact that any two facets are connected by a sequence of flips). A non irreducible subword complex is called \defn{reducible}.
\end{definition}

We will see below that every reducible subword complex is isomorphic to a subword complex of smaller rank (Corollary~\ref{cor:reducible}). This explains our choice of terminology. Before proving this, we need a notion of flats of a list of vectors in a vector space.

\begin{definition}
Let~$L=(v_1,\dots ,v_r)$ be a list of vectors (with possible repetitions) spanning a vector space~$V$. A \defn{flat}~$F$ of~$L$ is any sublist $F\subset L$ that can be obtained as the intersection $F=U\cap L$ for some subspace~$U\subset V$. 
\end{definition}

The flats of~$\Root{I}{Q}$ will be used to define the comultiplication of the Hopf algebra structure on subword complexes. The main ingredient in the definition is that every flat encodes the root function of a subword complex of smaller rank, which turns out to be isomorphic to the link of a face of the initial subword complex. This result, which we call the ``Decomposition theorem of subword complexes", has its origins in~\cite{ceballos_subword_2013} and was presented for finite types in a slightly weaker version in~\cite{PilaudStump-brickPolytopes}, see Remark~\ref{rem:decomposition_theorem}.

\subsection{Decomposition theorem of subword complexes}

Given a flat~$F$ of~$\Root{I}{Q}$ denote by $V_F\subset V$ the subspace of $V$ spanned by the roots in~$F$. This subspace contains a natural root system
\[
\Phi_F=\Phi_F^+\sqcup  \Phi_F^-
\] 
where $\Phi_F,\Phi_F^+,\Phi_F^-$ are the restrictions of $\Phi,\Phi^+,\Phi^-$ to $V_F$ respectively. We denote by $\Delta_F$ the corresponding set of simple roots and by $W_F$ the associated Coxeter group. In the case of infinite Coxeter groups, the fact that the root system intersected with a subspace is again a root system with simple roots contained in $\Phi^+$ is a non-trivial result by Dyer in~\cite{dyer_reflection_1990}. 
For convenience, denote by \[J_F=\{j_1,\dots , j_{r'} \} \subset [r]\] the set of positions in $Q$ whose corresponding roots~$\Root{I}{j_k}$ belong to~$F$. Define~$\beta_F=(\beta_1, \dots ,\beta_{r'})$ as the list of roots 
\[
\beta_k := \wordprod{B_{j_k}}(\alpha_{q_{j_k}}),
\]
where $B_{j} := ([j-1]\ssm I) \ssm J_F$ is the set of positions on the left of $j$ in the complement of $I$ whose corresponding roots are not in $F$.

\begin{lemma}\label{lem:simple_roots}
The roots $\beta_1,\dots ,\beta_{r'}$ are simple roots of the root system $\Phi_F$.
\end{lemma}

\begin{proof}
For every $\ell \in [r]$, consider the sets $A_\ell=[\ell-1]\ssm I$ and $B_\ell = ([\ell-1]\ssm I) \ssm J_F$ as above. 

Let $s_{\Root{I}{\ell}}\in W$ denote the reflection that is orthogonal to the root~$\Root{I}{\ell}= \wordprod{A_\ell}(\alpha_{q_\ell})$. Since this reflection can be written as the conjugate $\wordprod{A_\ell}{q_\ell} \wordprod{A_\ell}^{-1}$, one deduces the formula 

\begin{equation}\label{for:prod_Aj_Bj}
\wordprod{B_\ell} = \left( \prod_{k\in A_\ell \cap J_F} s_{\Root{I}{k}} \right) \wordprod{A_\ell}.
\end{equation}
Denote by~$P_k$ the subword of~$Q$ with positions in the set~$B_{j_k}$, for~$k\in[r']$. 
If $\ell \in B_{j_k}$, then its corresponding root in the list of inversions of $P_k$ is $\wordprod{B_\ell}(\alpha_{q_\ell})$, which by Equation~\eqref{for:prod_Aj_Bj} is equal to 
\[
\left( \prod_{k\in A_\ell \cap J_F} s_{\Root{I}{k}} \right) \Root{I}{\ell}. 
\]
 Since $k\in J_F$, all the terms $\Root{I}{k}$ in the expression belong to the flat $F$, while~$\Root{I}{\ell}$ does not. Therefore, non of the inversions of $P_k$ belong to the subspace $V_F$ spanned by $F$.
 On the other hand,
 \[
 \beta_k = P_k(\alpha_{q_{j_k}}) = \left( \prod_{k\in A_{j_k} \cap J_F} s_{\Root{I}{k}} \right) \Root{I}{j_k}. 
 \]
Since all the terms $\Root{I}{k}$ and $\Root{I}{j_k}$ are in $F$, the root $\beta_k \in V_F$. 
Thus, $\beta_k$ is the first root in the list of inversions of the word $P=(P_k, q_{j_k})$ that belongs to the root subsystem $\Phi_F\subset \Phi$. By Proposition~\ref{prop:first_inversion_in_F}, we deduce that $\beta_k$ is a simple root of $\Phi_F$.
\end{proof}

We will define a subword complex $\subwordComplex{Q_F}{\pi_F}$ and a facet $I_F$ associated to $F$. Denote by
\[
Q_F:=(q'_1,\dots , q'_{r'})
\]
the word whose letters are the generators of the Coxeter group~$W_F$ corresponding to the simple roots $\beta_1,\dots ,\beta_{r'}$. The set $I_F$ corresponding to $I$ is given by
\[
I_F := \{i\in [r'] : j_i \in I \},
\]
and the element~$\pi_F\in W_F$ is the product of the letters in the subword of $Q_F$ with positions at the complement of $I_F$. We also denote by $\bar I_F$ the face of $\subwordComplex{Q}{\pi}$ corresponding to $I_F$, or in other words, the elements $i\in I$ whose corresponding roots~$\Root{I}{i}$ belong to~$F$.

\begin{theorem}[Decomposition theorem of subword complexes]
\label{thm:docompositon}
Let $I$ be a facet of a subword complex $\subwordComplex{Q}{\pi}$ of (possibly infinite) type $W$. If $F$ is a flat of~$\Root{I}{Q}$, then $F$ is the root function of the subword complex $\subwordComplex{Q_F}{\pi_F}$ of type $W_F$ with respect to the facet $I_F$. Moreover,
\[
\subwordComplex{Q_F}{\pi_F} \cong \operatorname{Link}_{\subwordComplex{Q}{\pi}}(I\setminus \bar I_F).
\]
\end{theorem}

\begin{proof}
Let $k\in[r']$ be a fixed position in the word~$Q_F$. We need to show $\RootF{I_F}{k}=\Root{I}{j_k}$, 
where~$\mathsf{r}_F$ denotes the root function of the subword complex $\subwordComplex{Q_F}{I_F}$. 
For simplicity, let $P$ be the subword of $Q$ with positions at $[j_k-1]\ssm I$. This subword can be subdivided into parts 
\[
P=(P_1,p_{i_1},\ P_2,p_{i_2}, \dots,\ P_m,p_{i_m},\ P_{m+1}),
\]
where~$p_{i_1},\dots,p_{i_m}$ are the letters whose corresponding roots in the root function belong to the flat~$F$. The corresponding subword of~$Q_F$ is given by $P'=(p_1',\dots,p_m')$, where $p_\ell'$ is the reflection in $W_F$ orthogonal to $P_1P_2\dots P_\ell(\alpha_{p_{i_\ell}})$.
This reflection can be written as
\[
p_\ell' = P_1\dots P_\ell \ \  p_{i_\ell} \ P_\ell^{-1}\dots  P_1^{-1}.
\]
The root function associated to the flat $F$ can be then computed as
\begin{align*} 
\RootF{I_F}{k} &= p_1'\dots p_m' (\beta_k)  \\ 
 &=  \left( \prod_{\ell=1}^m P_1\dots P_\ell \ \  p_{i_\ell} \ P_\ell^{-1}\dots P_1^{-1} \right) \left(P_1\dots P_{m+1}(\alpha_{q_{j_k}}) \right) \\ 
 &= P_1\ p_{i_1}\ P_2\ p_{i_2}\ \dots\ P_m\ p_{i_m}\ P_{m+1}\ (\alpha_{q_{j_k}})  \\ 
 &=  \Root{I}{j_k}. 
\end{align*}
Therefore, the flat $F$ is the root function of the subword complex $\subwordComplex{Q_F}{\pi_F}$ with respect to the facet $I_F$. 
Note that the subword of $Q_F$ with positions at the complement of $I_F$ is a reduced expression of the element~$\pi_F\in W_F$. The reason is that the roots in its inversion set are all different. In fact, this inversion set is formed by the roots~$\Root{I}{j_k}$ for~$k\notin I_F$, which is a subset of the inversion set of the reduced expression of $\pi$ given by the subword of $Q$ with positions at the complement of~$I$. 

Finally, the faces in the link of~$I\setminus \bar I_F$ in $\subwordComplex{Q}{\pi}$ can be obtained from~$I$ by flipping positions whose corresponding roots belong to the flat~$F$. By Lemma~\ref{lem:subword_flips}, these flips only depend on the root function, and therefore are encoded by the root function of~$\subwordComplex{Q_F}{\pi_F}$. As a consequence,
\[
\subwordComplex{Q_F}{\pi_F} \cong \operatorname{Link}_{\subwordComplex{Q}{\pi}}(I\setminus \bar I_F).
\]
\end{proof}

\begin{remark}\label{rem:decomposition_theorem}
The Decomposition theorem of subword complexes (Theorem~\ref{thm:docompositon}) is a generalization of~\cite[Lemma~5.4]{ceballos_subword_2013}, which is a particular case for a family of subword complexes related to the $c$-cluster complexes in the theory of cluster algebras.
The result for finite types can be found in a slightly weaker version in~\cite[Proposition~3.6]{PilaudStump-brickPolytopes}\footnote{There is a small mistake in the statement of~\cite[Proposition~3.6]{PilaudStump-brickPolytopes}. The restricted subword complex is isomorphic to the link of a face of the initial subword complex, and not to the faces reachable from the initial facet as suggested.}, which is used as an inductive tool in the construction of Coxeter brick polytopes. 
The present version of the Theorem is stronger for two reasons. First, in~\cite[Proposition~3.6]{PilaudStump-brickPolytopes}, the restricted subword complex is not explicitly described but recursively constructed in the proof of the result by scanning the word from left to right, while the present version gives a precise description of the word $Q_F$ in the restricted subword complex. Second,~\cite[Proposition~3.6]{PilaudStump-brickPolytopes} was proven for finite Coxeter groups, while the present version works uniformly for arbitrary Coxeter groups, finite or not.

 \end{remark}

\begin{corollary}\label{cor:reducible}
A reducible subword complex is isomorphic to an irreducible subword complex of smaller rank.
\end{corollary}

\begin{proof}
Suppose that the root configuration $R(I)$ does not generate the space $V$ for some facet $I$. 
Note that
\[
\subwordComplex{Q}{\pi} \cong \subwordComplex{Q_F}{\pi_F}
\]
for the flat $F$ consisting of the roots that belong to the span of $R(I)$. Since $W_F$ is a Coxeter group of smaller rank and $\subwordComplex{Q_F}{\pi_F}$ is irreducible, the result follows.
\end{proof}

\begin{example}[Example~\ref{ex:A3k2} continued]
Consider the subword complex in Example~\ref{ex:A3k2} and the root function associated to the facet~$I=\{1,3,7,8,9,10\}$ (also illustrated in Figure~\ref{fig:A3k2_roots}):
\[
\begin{array}{rcccccccccccc}
Q=  & (\circled{\blue{s_1}},  & \red{s_2},  & \circled{\blue{s_3}},  & \red{s_1},  & \circled{\red{s_2}},  & \circled{\red{s_3}},  & \blue{s_1},  & \circled{\blue{s_2}},  & \circled{\blue{s_3}},  & \blue{s_1},  & \circled{\red{s_2}}, & \red{s_1})  \\
\Root{I}{Q}= & (\blue{\alpha_1},  & \red{\alpha_2},  & \blue{\alpha_{23}},  & \red{\alpha_{12}}, & \red{\alpha_1}, & \red{\alpha_{123}},  & \blue{-\alpha_2}, & \blue{\alpha_{23}},& \blue{-\alpha_{123}},  & \blue{-\alpha_2}, &\red{\alpha_{23}}, & \red{\alpha_3})
\end{array}
\]

\noindent
Let $F=(\blue{\alpha_1},\ \blue{\alpha_{23}},\ \red{\alpha_1},\ \red{\alpha_{123}},\ \blue{\alpha_{23}},\ \blue{-\alpha_{123}},\ \red{\alpha_{23}})$ be the flat at positions $J_F=\{1,3,5,6,8,9,11\}$. The list of beta simple roots and the associated word are
\[
\begin{array}{rccccccc}
 \beta_F= & (\blue{\alpha_1},  & \blue{\alpha_{23}}, & \red{\alpha_1},  & \red{\alpha_{23}}, & \blue{\alpha_1},  & \blue{\alpha_{23}}, & \red{\alpha_1} )    \\
 Q_F= & (\blue{s_1},  & \blue{s_{23}}, & \red{s_1},  & \red{s_{23}}, & \blue{s_1},  & \blue{s_{23}}, & \red{s_1} )   
\end{array}
\]

\noindent
There is one root in $\beta_F$ for each circled letter $s_j$ in $Q$. This root is computed by applying all the underlined red letters which are not circled on its left to $\alpha_j$. For example, for the fourth circled letter, which is an $s_3$ in this case, one gets the root $\beta_4=\red{s_2}\red{s_1}(\red{\alpha_3})=\red{\alpha_{23}}$. The restricted facet is $I_F=\{1,2,5,6\}$ and the element $\pi_F=s_1s_{23}s_1$. The Coxeter group $W_F$ is generated by the simple transpositions $\{s_1,s_{23}\}$, and turns out to be isomorphic to the symmetric group $\mathcal S_3$. Thus, $\subwordComplex{Q_F}{\pi_F}$ can be written as the type $A_2$ subword complex
\[
\subwordComplex{Q_F}{\pi_F} = \subwordComplex{(s_1,s_2,s_1,s_2,s_1,s_2,s_1)}{[3\ 2\ 1]}.
\]

\noindent
On the other hand, the link of $I\setminus \bar I_F$ in $\subwordComplex{Q}{\pi}$ is obtained by deleting the two non-underlined letters in Q which are not circled,  
\[
\operatorname{Link}_{\subwordComplex{Q}{\pi}}(I\setminus \bar I_F) = \subwordComplex{(s_1,s_2,s_3,s_1,s_2,s_3, -,s_2,s_3, -,s_2,s_1)}{[4\ 3\ 2\ 1]}.
\]

\noindent
This subword complex is not irreducible, since its root configuration spans only a 2-dimensional subspace. Moreover, it is isomorphic to $\subwordComplex{Q_F}{\pi_F}$ which is a subword complex of smaller rank. Note that these two complexes being isomorphic is not a straight forward fact without using the concept of root functions. For example, one can see that positions 2,4, and 10 in the second are non-vertices of the complex, as they appear in every reduced expression of $[4\ 3\ 2\ 1]$ in the word. Finding a characterization of the non-vertices of subword complexes is an open problem in general.

The tuple $(W_F,Q_F,\pi_F,I_F)$ is called a \defn{flat decomposition} of $(W,Q,\pi,I)$. All possible flat decompositions for the previous example are illustrated in Figure~\ref{fig:A3k2_flats}. The particular example we computed is the second from top to bottom in the middle column. The shaded examples are the non irreducible ones.

\begin{figure}[htbp]
\begin{center}
\includegraphics[width=0.95\textwidth]{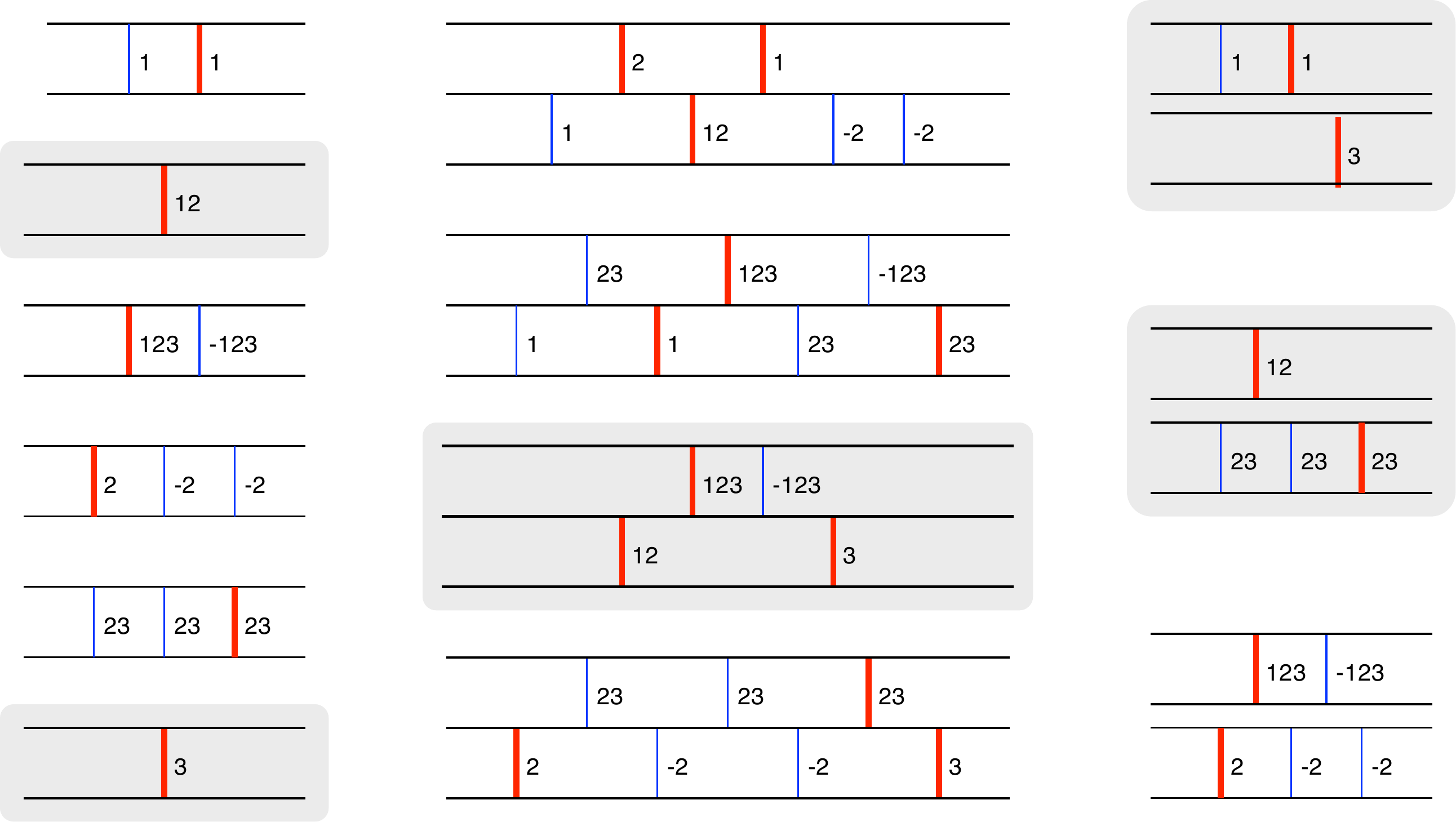}
\caption{Flat decompositions of the subword complex in Figure~\ref{fig:A3k2_roots} for the given facet.}
\label{fig:A3k2_flats}
\end{center}
\end{figure}

\end{example}

\begin{example}[Affine type $\tilde A_2$]
Let $W=\tilde A_2$ be the affine Coxeter group with simple generators $S=\{\tilde s_0,\tilde s_1,\tilde s_2\}$ satisfying $(\tilde s_0\tilde s_1)^3=(\tilde s_0\tilde s_2)^3=(\tilde s_1\tilde s_2)^3=Id$. Let $Q=(\tilde s_0,\tilde s_1,\tilde s_2,\tilde s_0,\tilde s_1,\tilde s_2,\tilde s_0,\tilde s_1,\tilde s_2)$ and $\pi=\tilde s_0\tilde s_1\tilde s_2\tilde s_0\tilde s_1\tilde s_2$. The subword complex $\subwordComplex{Q}{\pi}$ has 9 vertices and 6 maximal facets $\{1,2,3\},\{2,3,4\},\{3,4,5\},\{4,5,6\},\{5,6,7\},\{6,7,8\},\{7,8,9\}$. A picture of this subword complex is illustrated in Figure~\ref{fig:A2affine}.

\begin{figure}[htbp]
\begin{center}
\includegraphics[width=0.45\textwidth]{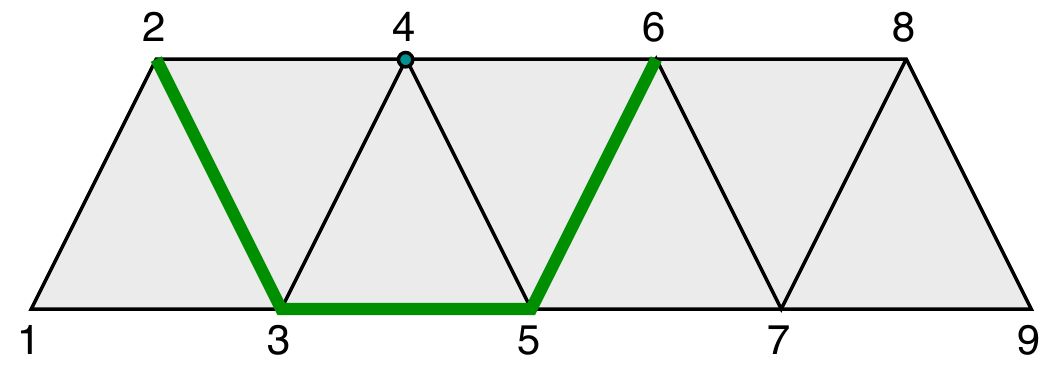}
\caption{A subword complex $\subwordComplex{Q}{\pi}$ of an affine Coxeter group of type $\tilde A_2$ for $Q=(\tilde s_0,\tilde s_1,\tilde s_2,\tilde s_0,\tilde s_1,\tilde s_2,\tilde s_0,\tilde s_1,\tilde s_2)$ and $\pi=\tilde s_0\tilde s_1\tilde s_2\tilde s_0\tilde s_1\tilde s_2$. The bold green sub-complex is the link of the vertex 4, and is isomorphic to a restricted subword complex described below. }
\label{fig:A2affine}
\end{center}
\end{figure}

\noindent
Let $\{\tilde \alpha_0,\tilde \alpha_1,\tilde \alpha_2\}$ be the simple roots of the root system of type $\tilde A_2$. The group acts on the roots according the following rule which is extended by linearity, 
\[
\tilde s_i(\tilde \alpha_j) =
\begin{cases}
  -\tilde \alpha_j &\text{if } i=j,\\
  \tilde \alpha_i+\tilde \alpha_j &\text{otherwise.}
\end{cases}
\]
The root function with respect to the facet $I=\{3,4,5\}$ is given by,
 \[
\begin{array}{rccccccccc}
Q=  & (\red{\tilde s_0},  & \circled{\red{\tilde s_1}},  & \circled{\blue{\tilde s_2}},  & \blue{\tilde s_0},  & \circled{\blue{\tilde s_1}},  & \circled{\red{\tilde s_2}},  & \red{\tilde s_0},  & \red{\tilde s_1},  & \red{\tilde s_2})  \\
\Root{I}{Q}= & (\red{\tilde \alpha_0},  & \circled{\red{\tilde \alpha_{01}}},  & \circled{\blue{\tilde \alpha_{0012}}},  & \blue{\tilde \alpha_1},  & \circled{\blue{-\tilde \alpha_{01}}},  & \circled{\red{\tilde \alpha_{0012}}},  & \red{\tilde \alpha_{00112}},  & \red{\tilde \alpha_{0001122}},  & \red{\tilde \alpha_{00011122} }),
\end{array}
\]
where the roots are represented by the subindices used when written as a linear combination of simple roots. For example, $\tilde \alpha_{0001122}=3\tilde \alpha_0+2\tilde \alpha_1+2\tilde \alpha_2$ since there are three 0's, two 1's, and two~2's. 

Let $F=({\red{\tilde \alpha_{01}}}, {\blue{\tilde \alpha_{0012}}},{\blue{-\tilde \alpha_{01}}}, {\red{\tilde \alpha_{0012}}})$ be the flat at positions $J_F=\{2,3,5,6\}$. 
The list of beta simple roots and the associated word are
\[
\begin{array}{rcccc}
 \beta_F= &  ({\red{\tilde \alpha_{01}}}, &{\blue{\tilde \alpha_{02}}},&{\blue{\tilde \alpha_{01}}},& {\red{\tilde \alpha_{02}}}) \\
 Q_F= & (\red{\tilde s_{01}},  & \blue{\tilde s_{02}}, & \blue{\tilde s_{01}},  & \red{\tilde s_{02}})   
\end{array}
\]
The restricted facet is $I_F=\{2,3\}$ and the element $\pi_F=\tilde s_{01} \tilde s_{02}$. The Coxeter group $W_F$ is generated by $\{s_{01},s_{02}\}$ and is isomorphic to a Coxeter group of type $A_2$.
The restricted subword complex is 
\[
\subwordComplex{Q_F}{\pi_F} = \subwordComplex{(\tilde s_{01},\tilde s_{02},\tilde s_{01},\tilde s_{02})}{\tilde s_{01}\tilde s_{02}}.
\]
It has 4 vertices and 3  one dimensional facets $\{1,2\},\{2,3\},\{3,4\}$.
The face corresponding to $I_F$ in the original subword complex is $\bar I_F=\{3,5\}$. 
The set $I\smallsetminus \bar I_F = \{4\}$ and the $\operatorname{Link}_{\subwordComplex{Q}{\pi}}(I\setminus \bar I_F)$ is the link of vertex 4 in Figure~\ref{fig:A2affine}, which has 4 vertices and 3 one dimensional facets $\{2,3\},\{3,5\},\{5,6\}$.
This verifies that $\subwordComplex{Q_F}{\pi_F} \cong \operatorname{Link}_{\subwordComplex{Q}{\pi}}(I\setminus \bar I_F)$ as implied by Theorem~\ref{thm:docompositon}.


\end{example}

%
%
%
%

\section{A Hopf algebra of subword complexes}\label{sec:Hopf}

Let~$Y_n$ be the set of equivalent classes of tuples~$(W,Q,\pi,I)$ where $W$ is a (possibly infinite) Coxeter group of rank~$n$, and~$I$ is a facet of an irreducible subword complex~$\subwordComplex{Q}{\pi}$. For $n=0$, by convention, we assume there is a unique empty tuple ${\bf 1}=(W_0,\emptyset,\emptyset,\emptyset)$ and in particular $|Y_0|=1$.
Two tuples are considered to be equivalent, denoted by $(W,Q,\pi,I)\cong (W',Q',\pi',I')$,
if and only if there is an isomorphism $\phi:W\rightarrow W'$ which maps generators of $W$ to generators of $W'$ such that $\phi(Q)=Q'$ up to commutation of consecutive commuting letters, $\phi(\pi)=\pi'$ and $I'$ are the positions in $Q'$ that correspond (up to the performed commutations) to the positions of $I$ in $\phi(Q)$. 
Note that such commutations only alter the  subword complexes by relabelling of its vertices.

The main result of this section is to show that the graded vector space
\[
\YY:= \bigoplus_{n\geq 0}  \ k[Y_n]
\]
may be equipped with a structure of connected graded Hopf algebra. We recommend the reader to \cite{AguiarBergeronSottile} for more on connected graded Hopf algebra's axioms.

\begin{remark}
Note that $k[Y_n]$ is infinite dimensional. In most situations, we need finite dimensional subspaces compatible with the Hopf structure. For this, we introduce a double filtration of the spaces $k[Y_n]$. Let
  $$ k[Y_n] = \bigcup_{m\ge 2 \atop \ell\ge 1} k[Y_n^{m,\ell}]$$
  where $Y_n^{m,\ell}$ is the set of equivalent classes of tuples~$(W,Q,\pi,I)$ such that
  $W$ is a of rank~$n$, $Q$ is of length $\le \ell$, and for any two generators $s_i,s_j\in S$,
  the smallest $m_{ij}$ such that $(s_is_j)^{m_{ij}}=1$ satisfies $m_{ij}\le m$ or $m_{ij}=\infty$.
  We now have that $Y_n^{m,\ell}$ is finite, hence $k[Y_n^{m,\ell}]$ is finite dimensional.
  Moreover
   $$k[Y_n^{m,\ell}] \subseteq k[Y_n^{m+1,\ell}] \quad \text { and } \quad k[Y_n^{m,\ell}]\subseteq k[Y_n^{m,\ell+1}].$$
   We will see that this filtration is compatible with the Hopf structure that we introduce.
\end{remark}

\subsection{Comultiplication $\Delta\colon \YY\to\YY\otimes\YY$ and counit $\epsilon\colon\YY\to k$.}

\begin{definition} Let $V$ denote the space generated by $\Roots{I}$.
A \defn{k-flat-decomposition} of $\Roots{I}$ is a $k$-tuple of flats $(F_1,F_2,\ldots,F_k)$ such that the $F_i$'s are irreducible flats of $\Root{I}{Q}$, that is the space $V_{F_i}$ 
spanned by $F_i$ is the same as the space spanned by the roots in  $\Roots{I_{F_i}}$, and  we also require that
$V=V_{F_1} \oplus V_{F_2} \oplus \cdots \oplus V_{F_k}$.
\end{definition}

\begin{definition}\label{def:comulti}
The subword complex \defn{comultiplication} of a tuple $(W,Q,\pi,I)$ is defined as 
\[
\Delta(W,Q,\pi,I) :=  \sum_{(F_1,F_2)} (W_{F_1},Q_{F_1}, \pi_{F_1}, I_{F_1}) \otimes (W_{F_2},Q_{F_2}, \pi_{F_2}, I_{F_2})\,,
\]
where the sum is over all 2-flat-decompositions $(F_1,F_2)$ of $\Roots{I}$. The map $\Delta$ is then extended to $\YY$ by linearity. 
\end{definition}

The comultiplication $\Delta$ is clearly coassociative since both $(Id\otimes\Delta)\Delta$ and $(\Delta\otimes Id)\Delta$ depend only on $\text{3-flat-decompositions of } \Roots{I}$.
Furthermore, we have that $\Delta$ is graded since for any $\text{2-flat-decomposition}$~$(F_1,F_2)$ of $\Roots{I}$, we have that the dimensions of the flats add to the dimension of $\Roots{I}$.
In addition, $(F_2,F_1)$ is as well a 2-flat decomposition of $R(I)$, which makes $\Delta$ a cocommutative operation. Remark that the length of $Q$ is greater than or equal to the sum of the length of $Q_{F_1}$ and $Q_{F_2}$. Now, if we take two generators $s_i,s_j$ of $W_{F_1}$, then either $m_{ij}=\infty$ or $m_{ij}=m_{i'j'}$ for some generators $s_{i'},s_{j'}$ of $W$ (This follows from~\cite[Thm 4.5.3]{Bjorner_Brenti}). Thus we have
  $$\Delta\colon k[Y_n^{m,\ell}] \longrightarrow \bigoplus_{n_1+n_1=n \atop \ell_1+\ell_2=\ell} 
     k[Y_{n_1}^{m,\ell_1}] \otimes k[Y_{n_2}^{m,\ell_2}] \,.$$

The counit for $\Delta$ is given by $\epsilon\colon\YY\to k$ where $\epsilon(Y_n)=0$ for $n>0$ and   $\epsilon({\bf 1})=1$.
This map clearly satisfies the axioms of a counit. This makes $(\YY,\Delta,\epsilon)$ a graded cocommutative coalgebra.
 
\begin{example}[Example~\ref{ex:A2} continued]
\label{ex:A2_comultiplication}
The comultiplication of the facet $I=\{2,3,5\}$ of the the subword complex $\subwordComplex{(s_1,s_2,s_1,s_2,s_1)}{s_1s_2}$ in Example~\ref{ex:A2} is given in Figure~\ref{fig:A2_comultiplication}. Recall that the indices on the right of each commutator represent the indices of the corresponding root in the root function. For example 12 represents the root $\alpha_{12}=\alpha_1+\alpha_2$. 

\begin{figure}[htbp]
\begin{center}
\includegraphics[width=0.98\textwidth]{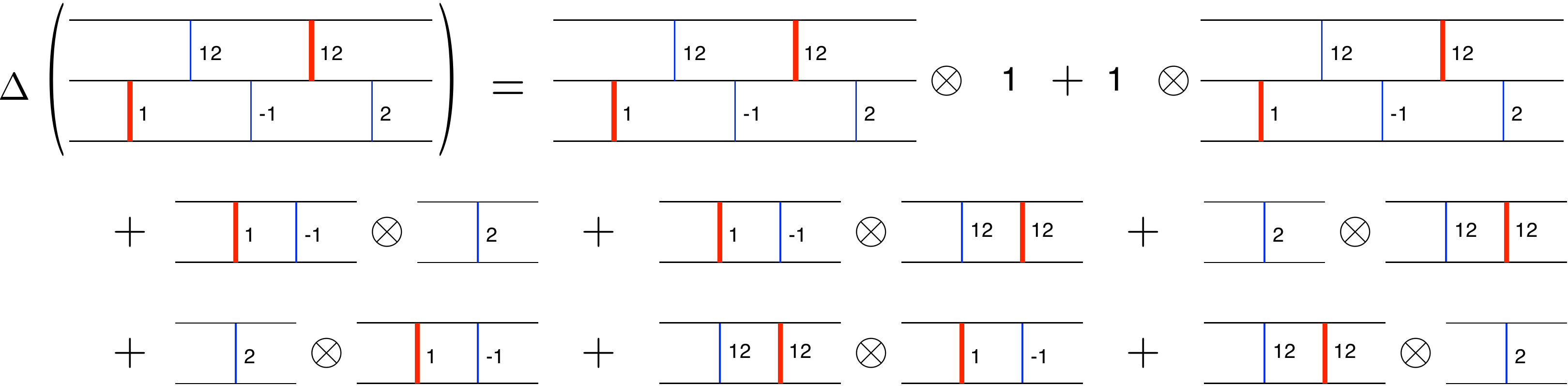}
\caption{Example of the subword complex comultiplication.}
\label{fig:A2_comultiplication}
\end{center}
\end{figure}
\end{example}

\subsection{Multiplication $m\colon \YY\otimes \YY\to \YY$ and unit $u\colon k\to \YY$.}

Let~$W$ and $W'$ be two Coxeter groups with generating sets~$S=\{s_1,\dots, s_n\}$ and~$S'=\{s_1',\dots,s_n'\}$ respectively. Let~$\Phi$ and ~$\Phi'$ be two associated root systems with simple roots~$\Delta=\{\alpha_s \ |\  s\in S\}$ and~$\Delta'=\{\alpha_{s'} \ |\  s'\in S'\}$. We denote by~$\overline W=W\times W'$ the \defn{augmented Coxeter group} generated by the disjoint union~$S \sqcup S'$, where the generators of~$S$ are set to commute with all the generators of~$S'$. In other words, the Coxeter graph of~$\overline W$ is the union of the Coxeter graphs of~$W$ and~$W'$. The corresponding \defn{augmented root system} with simple roots~$\Delta \sqcup \Delta'$ is denoted by~$\overline \Phi$. 

Throughout this section, the word~$Q=(q_1,\dots,q_r)$ will denote a word in~$S$, $\pi$ an element of~$W$, and~$I\subset [r]$ a facet of the subword complex~$\subwordComplex{Q}{\pi}$. Similarly,~$Q'=(q_1',\dots,q_{r'}')$, $\pi'\in W'$ and $I'\subset [r']$ will denote their analogous for the Coxeter system $(W',S')$. 

\begin{definition}\label{def:multi}
The subword complex \defn{multiplication} between two tuples is defined by
\[
(W,Q,\pi,I) \cdot (W',Q',\pi',I') = (\overline W,QQ' , \pi\pi', \shifted{II'}),
\]
were $QQ'=(q_1,\dots,q_r,q_1',\dots,q_{r'}')$ denotes the concatenation of $Q$ and $Q'$, $\pi\pi'$ is the product of $\pi$ and $\pi'$ in $\overline W$, and~$\shifted{II'}$ denotes the {\sl shifted} facet 
\[
\shifted{II'} := \{k\in [r+r'] \ : \  k=i \text{ for some } i\in I \text{, or } k=i'+r \text{ for some } i'\in I' \}
\]
The map $m$ is then extended to $\YY\otimes\YY$ by linearity. 
\end{definition}

Since concatenation and union are associative operations, the multiplication $m$ is associative as well. 
The multiplication $m$ is graded as the dimensions are added in the result. The element $\bf 1$ is the unit for the multiplication,
hence the map $u\colon k\to \YY$ defined by $u(1)={\bf 1}$ is the unit map for~$m$. We have that $(\YY,m,u)$ is a graded algebra. Moreover, this algebra is commutative since~$QQ'=Q'Q$ up to commutation of consecutive commuting letters and $\pi\pi'=\pi'\pi$. Remark now that the length of $QQ'$ is the sum of the lengths of $Q$ and $Q'$. Also, for any generators $s_i,s_j$ of $\overline W$, they are either both in $W$ or both in $W'$, or one in each. This shows that
 $$ m\colon k[Y_{n_1}^{m_1,\ell_1}] \otimes k[Y_{n_2}^{m_2,\ell_2}] \longrightarrow
     k[Y_{n_1+n_2}^{\max\{m_1,m_2\},\ell_1+\ell_2}]\,. $$
     
\subsection{Hopf structure.} 

\begin{theorem}
The graded vector space
\[
\YY := \bigoplus_{n\geq 0}  \ k[Y_n]
\]
equipped with the subword complex multiplication and comultiplication is a connected graded Hopf algebra. This Hopf algebra is commutative and cocommutative.
\end{theorem}

\begin{proof} We only have to show that the structure $(\YY, m, u, \Delta, \epsilon)$ gives a connected graded bialgebra. In such case, the antipode $S\colon \YY\to\YY$ is uniquely determined as in \cite{Ehrenborg, Takeuchi} (see Section~\ref{sec:antipode}). This will give us that $(\YY, m, u, \Delta, \epsilon,S)$ is indeed a Hopf algebra.

We are already given that $(\YY,m,u)$ is a graded algebra and $(\YY,\Delta,\epsilon)$ is a graded coalgebra. 
It is connected since $|Y_0|=1$. Hence, we only need to show that $\Delta$ and $\epsilon$ are morphisms of algebras.
For $\epsilon\colon \YY\to k$, it is clear since $m({\bf 1}\otimes{\bf 1})=1$.
For $\Delta\colon \YY\to\YY\otimes\YY$ we have to show that
  $$ \Delta(xy) = \Delta(x) \Delta(y)$$
  where the multiplication on $\YY\otimes\YY$ is given by $(m\otimes m)(Id\otimes\tau\otimes Id)$ with $\tau(a\otimes b)=b\otimes a$. Let 
  $$ (\overline W,QQ' , \pi\pi', \shifted{II'}) = (W,Q,\pi,I) \cdot (W',Q',\pi',I') $$
as in Definition~\ref{def:multi}. We have that $ \Delta(\overline W,QQ' , \pi\pi', \shifted{II'})$ is given by
\begin{equation}\label{eq:lhs}
\sum_{(\overline F_1, \overline F_2)} (\overline W_{\overline F_1},QQ'_{\overline F_1}, \pi\pi'_{\overline F_1}, II'_{\overline F_1}) \otimes (\overline W_{\overline F_2},QQ'_{\overline F_2}, \pi\pi'_{\overline F_2}, II'_{\overline F_2}),
\end{equation}
where $(\overline F_1, \overline F_2)$ is 2-flat-decomposition of $\Roots{II'}$. Now we recall that $\overline W$ is $W\times W'$ and in particular the roots of $W$ and $W'$ are pairwise orthogonal in $\Roots{II'}$. This implies that $\overline F_1= F_1 \oplus F'_1$ and $\overline F_2= F_2 \oplus F'_2$ where $(F_1,F_2)$ and $(F'_1,F'_2)$ are 2-flat-decompositions of $\Roots{I}$ and $\Roots{I'}$ respectively. Moreover $F_1$ and $F'_1$ are orthogonal and the same hold for $F_2$ and $F'_2$.
If we now compute $\Delta(W,Q,\pi,I)\Delta(W',Q',\pi',I')$ we obtain
\begin{equation}\label{eq:rhs}
\begin{array}{rl}
\displaystyle \sum_{(F_1,F_2)  \atop (F'_1,F'_2)} 
 &\big((W_{ F_1},Q_{ F_1}, \pi_{ F_1}, I_{ F_1}) \otimes (W_{ F_2},Q_{ F_2}, \pi_{ F_2}, I_{ F_2})\big)\times \cr 
 &\qquad \qquad \qquad \big((W'_{ F'_1},Q'_{ F'_1}, \pi'_{ F'_1}, I'_{ F'_1}) \otimes (W'_{ F'_2},Q'_{ F'_2}, \pi'_{ F'_2}, I'_{ F'_2})\big)
\end{array}
\end{equation}
\begin{equation}\label{eq:rhs2}
\begin{array}{rl}
= \displaystyle \sum_{(F_1\oplus F'_1,F_2\oplus F'_2)} 
&(W_{ F_1}\times W'_{ F'_1},Q_{ F_1}Q'_{ F'_1}, \pi_{ F_1} \pi'_{ F'_1}, I_{ F_1}I'_{ F'_1}) \otimes \cr 
&\qquad \qquad \qquad  (W_{ F_2}\times W'_{ F'_2},Q_{ F_2}Q'_{ F'_2}, \pi_{ F_2}\pi'_{ F'_2}, I_{ F_2}I'_{ F'_2}),
\end{array}
\end{equation}
where $(F_1,F_2)$ and $(F'_1,F'_2)$ are as above.  Comparing Eq~\eqref{eq:lhs} and Eq~\eqref{eq:rhs2} we have, for $i=1,2$,
$$ \overline W_{\overline F_i}=W_{ F_i}\times W'_{ F'_i}; \quad QQ'_{\overline F_i}=Q_{ F_i}Q'_{ F'_i};  \quad \pi\pi'_{\overline F_i}= \pi_{ F_i} \pi'_{ F'_i}; \quad II'_{\overline F_i} =I_{ F_i}I'_{ F'_i}\,.$$
These equalities follow from orthogonality of  $\overline F_i= F_i \oplus F'_i$. 
\end{proof}

\section{Antipode}\label{sec:antipode}
\subsection{Takeuchi's formula} The antipode $S\colon \YY\to\YY$ exists and is unique. But its construction from \cite{Takeuchi} is certainly not cancelation free. It is always an interesting question to give an explicit cancelation free formula for the antipode. Takeuchi's formula \cite{Takeuchi} gives that for $\psi\in k[Y_n]$,
 \begin{equation}\label{TakeuchisFormula}
 S(\psi)=\sum_{\alpha\models n} (-1)^{\ell(\alpha)} m_\alpha\circ \Delta_\alpha(\psi),
 \end{equation}
where the sum is over all compositions $\alpha=(\alpha_1,\ldots,\alpha_\ell)$ of positive integers summing to $n$. Here $\ell(\alpha)$ is the number of parts of $\alpha$,
  $$m_\alpha\colon\ k[Y_{\alpha_1}]\otimes \cdots \otimes k[Y_{\alpha_\ell}]\longrightarrow k[Y_n]$$
  is the (composite) multiplication, and
  $$\Delta_\alpha=(\pi_{\alpha_1}\otimes \cdots \otimes \pi_{\alpha_\ell})\circ\Delta^{(\ell)}\colon  k[Y_n]\longrightarrow k[Y_{\alpha_1}]\otimes \cdots \otimes k[Y_{\alpha_\ell}]$$
  is the (co-composite) comultiplication defined with the projections $\pi_m\colon \YY\to k[Y_m]$.
  Remark that even if $k[Y_n]$ is infinite dimensional, the maps $m_\alpha\circ \Delta_\alpha$ restrict to the finite dimensional subspaces of $k[Y_n^{m,\ell}]$. Hence the antipode is well defined in equation~\eqref{TakeuchisFormula}.
In Example~\ref{ex:A2_comultiplication}, Takeuchi's formula gives the formula in Figure~\ref{fig:A2_antipode}. Note that in this particular case there are no cancelations. However, Takeuchi's formula has in general a lot of cancelations. In this section, we deduce from Takeuchi's formula an expression that is cancelation free. The coefficients in our formula count the number of acyclic orientations of certain graphs.

\begin{figure}[htbp]
\begin{center}
\includegraphics[width=0.98\textwidth]{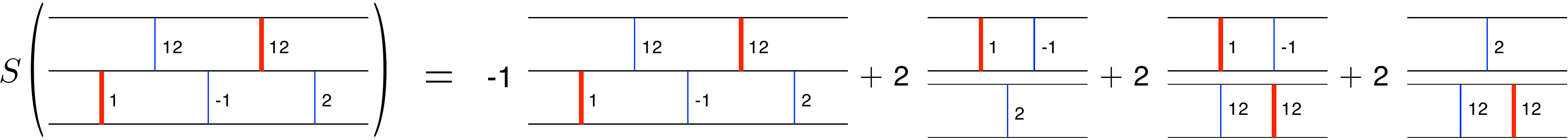}
\caption{Example of the antipode from Takeuchi's formula.}
\label{fig:A2_antipode}
\end{center}
\end{figure}

\noindent
Given $\psi= (W,Q,\pi,I)$ and $\alpha= (\alpha_1,\ldots,\alpha_\ell)$, we want to describe $m_\alpha\circ \Delta_\alpha(\psi)$. For this we need  to consider $\ell$-flat-decompositions ${\mathcal F}=(F_1,F_2,\ldots,F_\ell)$ of $\Roots{I}$ such that the dimension of the space~$U_i$ spanned by $F_i$ is equal to $\alpha_i$. Then 
  $$m_\alpha\circ \Delta_\alpha(\psi)= \sum_{{\mathcal F}=(F_1,F_2,\ldots,F_\ell) \atop \dim(U_i)=\alpha_i} (W_{\mathcal F},Q_{\mathcal F},\pi_{\mathcal F},I_{\mathcal F})\,,$$
where
  $$ W_{\mathcal F}=W_{F_1}\times \cdots \times W_{F_\ell}\,;\qquad Q_{\mathcal F}=Q_{F_1}\cdots Q_{F_\ell}\,;$$
  $$ \pi_{\mathcal F}=\pi_{F_1}\cdots \pi_{F_\ell}\,;\qquad \qquad I_{\mathcal F}=I_{F_1}\cdots I_{F_\ell}\,.$$
Let $\mathcal FD$ be the set of $\ell$-flat-decompositions of $\Roots{I}$  for $1\le \ell\le n$. For ${\mathcal F}=(F_1,F_2,\ldots,F_\ell)\in {\mathcal FD}$
we denote by $\ell({\mathcal F})$ the length $\ell$ of ${\mathcal F}$.
Then Takeuchi's formula~\eqref{TakeuchisFormula} can be written as
 \begin{equation}\label{TakeuchisFormula2}
 S(\psi)=\sum_{{\mathcal F}\in {\mathcal FD}} (-1)^{\ell({\mathcal F})}  (W_{\mathcal F},Q_{\mathcal F},\pi_{\mathcal F},I_{\mathcal F})\,.
 \end{equation}

To resolve the cancelations in \eqref{TakeuchisFormula2}, we will construct a sign reversing involution on the set ${\mathcal FD}$ where the sign of ${\mathcal F}\in {\mathcal FD}$
is given by $(-1)^{\ell({\mathcal F})}$.  We want our involution $\Phi\colon{\mathcal FD}\to {\mathcal FD}$ to map $\Phi({\mathcal F})={\mathcal F}'$ in such a way that
\begin{enumerate}
\item[(a)] ${\mathcal F} = {\mathcal F}'$, or
\item[(b)] ${\mathcal F} \ne {\mathcal F}'$ and $(W_{\mathcal F},Q_{\mathcal F},\pi_{\mathcal F},I_{\mathcal F})\cong (W_{\mathcal F'},Q_{\mathcal F'},\pi_{\mathcal F'},I_{\mathcal F'})$
and $\ell({\mathcal F}) = \ell({\mathcal F}')\pm 1$.
\end{enumerate}
To state our result and describe the desired involution, we need to introduce some objects associated to every ${\mathcal F}\in {\mathcal FD}$.

\subsection{Combinatorics of ${\mathcal FD}$.} \label{sec:comb}
We now define some relations on ${\mathcal FD}$ and will use that to define a directed graph associated to any ${\mathcal F}\in {\mathcal FD}$.
In the following we say that ${\mathcal F}\cong{\mathcal F'}$ if  and only if $(W_{\mathcal F},Q_{\mathcal F},\pi_{\mathcal F},I_{\mathcal F})\cong (W_{\mathcal F'},Q_{\mathcal F'},\pi_{\mathcal F'},I_{\mathcal F'})$.

To start, we say that ${\mathcal F}\le{\mathcal F'}$
if ${\mathcal F}\cong{\mathcal F'}$ and ${\mathcal F'}$ refines the parts of ${\mathcal F}$. That is for ${\mathcal F}=(F_1,F_2,\ldots,F_\ell)$ and ${\mathcal F'}=(F'_1,F'_2,\ldots,F'_m) $ we have $m\ge \ell$ and there exists $1\le i_1<i_2<\ldots <i_{\ell}=m$ such that
  $$ F_1=F'_1\cup\cdots\cup F'_{i_1}; \qquad F_2=F'_{i_1+1}\cup\cdots\cup F'_{i_2}; \quad\cdots\quad F_\ell=F'_{i_{ell-1}+1}\cup\cdots\cup F'_{m}.
$$
Given ${\mathcal F}\in {\mathcal FD}$, all the maximal refinement ${\mathcal F'}\ge{\mathcal F}$ will have the same length. 
We now state a few elementary lemmas. 

\begin{lemma} \label{lemma:Fsplit}
For $(F_1,F_2,\ldots,F_\ell)\le(F'_1,F'_2,\ldots,F'_m)$ in ${\mathcal FD}$, if $F'_i\cup F'_j\subseteq F_k$ then $F'_i$ and $F'_j$ are orthogonal in $\Roots{I}$.
\end{lemma}
\begin{proof}
Since $(F_1,F_2,\ldots,F_\ell)\le(F'_1,F'_2,\ldots,F'_m)$ implies that $(F_1,F_2,\ldots,F_\ell)\cong(F'_1,F'_2,\ldots,F'_m)$, we have that
 $W_{F_k} \cong W_{F'_i}\times W_{F'_j}\times \cdots$. In particular $F'_i$ and $F'_j$ are orthogonal in $\Roots{I}$.
\end{proof}

\begin{lemma} \label{lemma:lengthFmax}
For any ${\mathcal F},{\mathcal F}'\in {\mathcal FD}$ maximal refinement, if ${\mathcal F}\cong{\mathcal F'}$, then $\ell({\mathcal F})=\ell({\mathcal F'})$.
\end{lemma}
\begin{proof}
Suppose that $\ell({\mathcal F})>\ell({\mathcal F'})$, that would implies that a component of ${\mathcal F'}$ can be refine and this contradict the maximality of ${\mathcal F'}$.
Similarly the maximality of ${\mathcal F}$ implies  $\ell({\mathcal F})\not<\ell({\mathcal F'})$.
\end{proof}

\begin{lemma} \label{lemma:Fpermuted}
For any ${\mathcal F}\in {\mathcal FD}$ with $\ell({\mathcal F})=\ell$ and for any permutation $\sigma\colon\{1,\ldots,\ell\}\to\{1,\ldots,\ell\}$ we have
  $$\sigma{\mathcal F}=(F_{\sigma(1)},F_{\sigma(2)},\ldots,F_{\sigma(\ell)})\in {\mathcal FD} \qquad\text{and}\qquad \sigma{\mathcal F}\cong{\mathcal F} \,.$$
\end{lemma}
\begin{proof}
It is clear that permuting the entries of a $\ell$-flat decomposition is also a $\ell$-flat decomposition that is isomorphic.
\end{proof}

Let $S_\ell$ denote the permutation group of permutations  $\sigma\colon\{1,\ldots,\ell\}\to\{1,\ldots,\ell\}$. Given ${\mathcal F}\in {\mathcal FD}$, let
  $$\Psi({\mathcal F}) = \big\{ \sigma{\mathcal F'} : {\mathcal F'}\ge {\mathcal F} \text{ maximal refinement, } \ell=\ell({\mathcal F'}), \  \sigma\in S_\ell\big\}\subset {\mathcal FD}.$$
It is clear that for ${\mathcal F},{\mathcal K}\in {\mathcal FD}$ if $\Psi({\mathcal F})=\Psi({\mathcal K})$ then ${\mathcal F}\cong {\mathcal K}$. We need to pick consistently,
once and for all, an element $\Psi^0({\mathcal F})\in\Psi({\mathcal F})$ in each $\Psi({\mathcal F})$ in such a way that 
    $$\Psi({\mathcal F})=\Psi({\mathcal K})\quad\iff \quad\Psi^0({\mathcal F})=\Psi^0({\mathcal K})$$
for all ${\mathcal F},{\mathcal K}\in {\mathcal FD}$.    It is clear we can make such choices.

Our next task is to construct an oriented graph $G({\mathcal F})$ associated to every ${\mathcal F}\in {\mathcal FD}$. 
The orientation of the edges of $G({\mathcal F})$ depends on the choice of $\Psi^0({\mathcal F})$. This is why it is important to fix these choices consistently as above.
Let $\Psi^0({\mathcal F})=(F^0_1,F^0_2,\ldots,F^0_\ell)$ and ${\mathcal F}=(F_1,F_2,\ldots,F_m)$. 
Since $\Psi^0({\mathcal F})$ is the permutation of a maximal refinement of ${\mathcal F}$, there is a well defined map
$f_{\mathcal F}\colon\{1,\ldots,\ell\}\to\{1,\ldots,m\}$ such that $F^0_{i}\subseteq F_{f_{\mathcal F}(i)}$.
The graph $G({\mathcal F})=(V,E)$ is a graph on the vertices $V=\{1,2,\ldots,\ell\}$ where we have a directed edge 
 $$(i,j)\in E \quad\iff\quad\text{$F^0_i,F^0_j$ are {\bf not} othogonal and $f_{\mathcal F}(i)<f_{\mathcal F}(j)$.}$$
We now remark that Lemma~\ref{lemma:Fsplit} implies that if $F^0_i,F^0_j$ are not othogonal, then $f_{\mathcal F}(i)\ne f_{\mathcal F}(j)$. Hence there 
is an edge in $G({\mathcal F})$ whenever $F^0_i,F^0_j$ are not othogonal and the orientation of the edge is determined by the choice of $\Psi^0({\mathcal F})$.
We denote by $\overline{G}({\mathcal F})=(V,\overline{E})$ the simple graph obtained from $G({\mathcal F})=(V,E)$ by forgetting the orientation of the edges in $E$.
The graph $G({\mathcal F})$ is an {\sl acyclic orientation} of $\overline{G}({\mathcal F})$.

\begin{lemma} \label{lemma:graph}
If ${\mathcal F}\le{\mathcal F}'$, then $G({\mathcal F})=G({\mathcal F'})$.
\end{lemma}
\begin{proof}
Lemma~\ref{lemma:Fsplit} guaranties that $\overline{G}({\mathcal F})=\overline{G}({\mathcal F'})$. Since the refinement relation only groups consecutive parts of ${\mathcal F'}$ that are orthogonal (no edges), then
we have that the orientations is the same on both $G({\mathcal F})$, and $G({\mathcal F'})$.
\end{proof}

For any ${\mathcal F}\in {\mathcal FD}$ and  any acyclic orientation $G=(V,E)$ of $\overline{G}\big(\Psi^0({\mathcal F})\big)$, let $\ell=\ell\big(\Psi^0({\mathcal F})\big)$. We construct a unique permutation $\sigma_{\mathcal F}\in S_\ell$ such that $G\big(\sigma_{\mathcal F}\Psi^0({\mathcal F})\big)=G$. The desired permutation $\sigma_{\mathcal F}$ is constructed recursively as follows:
Let $V_1=V=\{1,\ldots,\ell\}$. The value $\sigma_{\mathcal F}(1)$ is the largest value of $V_1$ that is a source in $G$ restricted to $V_1$.
For $1<i\le\ell$, we let $V_i=V_{i-1}\setminus\{\sigma_{\mathcal F}(i-1)\}$ and we let $\sigma_{\mathcal F}(i)$ be the largest  value of $V_i$ that is a source in $G$ restricted to $V_i$.

\begin{remark} Let us summarize the objects associated to any ${\mathcal F}=(F_1,F_2,\ldots,F_m)\in {\mathcal FD}$.
We have picked consistently  $\Psi^0({\mathcal F})=(F^0_1,F^0_2,\ldots,F^0_\ell)$ among all the permutations of the parts of a maximal refinement of ${\mathcal F}$.
Using $\Psi^0({\mathcal F})$ we construct a well defined map $f_{\mathcal F}\colon\{1,\ldots,\ell\}\to\{1,\ldots,m\}$ which allow us to construct an oriented acyclic graph $G({\mathcal F})$.
The simple graph $\overline{G}({\mathcal F})$ contains edge between $i$ and $j$ if and only if the part $F^0_i$ and $F^0_j$ are not orthogonal.
Given any acyclic orientation $G$ of $\overline{G}({\mathcal F})$, there exists a unique permutation $\sigma_{\mathcal F}$ such that $G\big(\sigma_{\mathcal F}\Psi^0({\mathcal F})\big)=G$.
Remark that if ${\mathcal K}=\sigma\Psi^0({\mathcal F})$ for some ${\mathcal F}\in {\mathcal FD}$ and a permutation $\sigma$, then $f_{\mathcal K}=\sigma^{-1}$.
\end{remark}

\subsection{Antipode (cancelation free formula).} \label{sec:cancel free}
We now have the necessary combinatorial tools to state and prove the cancelation free formula for the antipode. The coefficients in our formula involve the number $a(\overline{G})$ of acyclic orientations of a simple graph $\overline{G}$.

\begin{figure}[t]
\begin{center}
\includegraphics[width=0.7\textwidth]{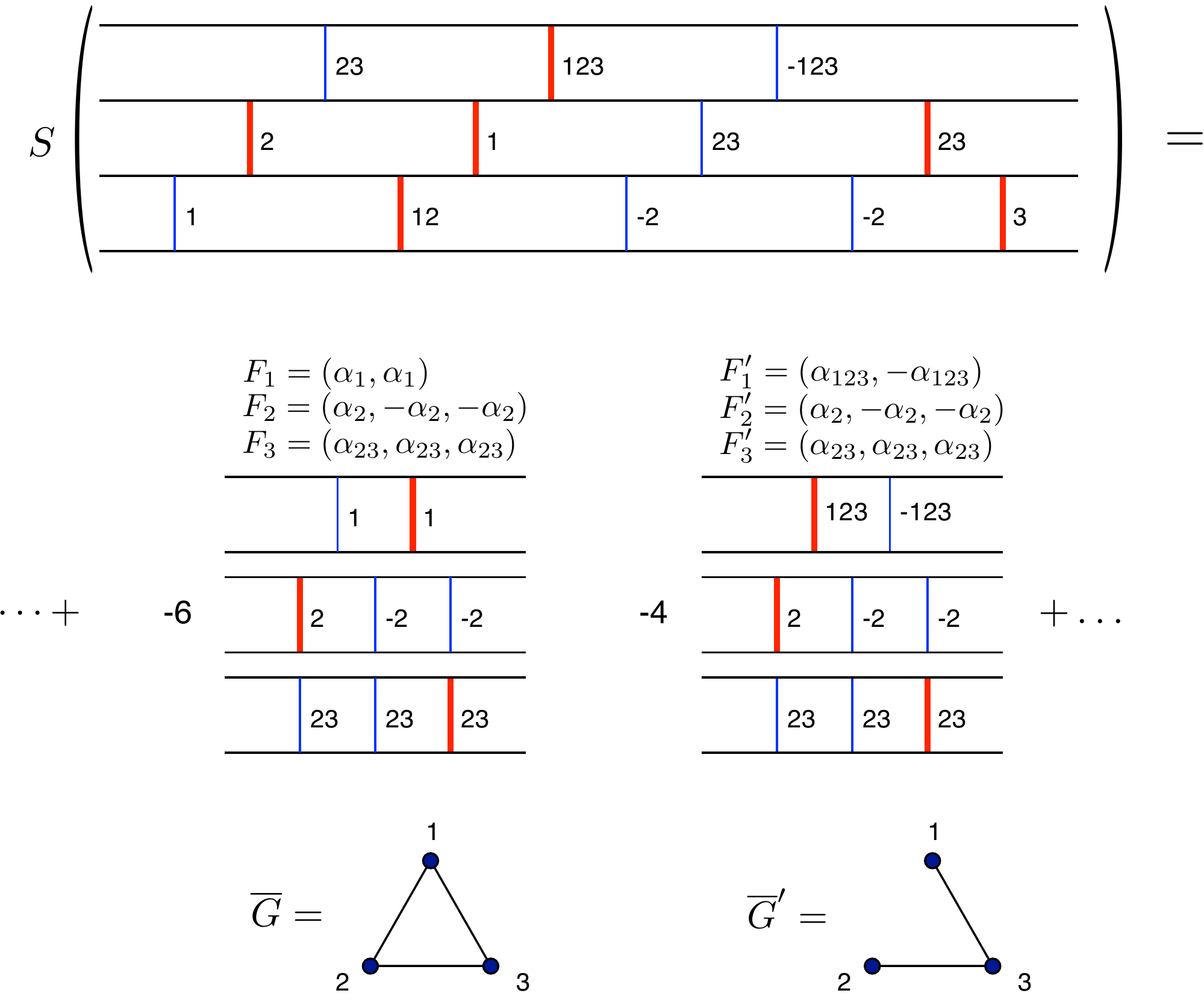}
\caption{Example of the cancellation free formula for the antipode. The number of acyclic orientations of the graph $\overline G= \overline G ((F_1,F_2,F_3))$ is 6, while the number of acyclic orientations of $\overline G'= \overline G ((F_1',F_2',F_3'))$ is 4.}
\label{fig:A3k2_cycle}
\end{center}
\end{figure}

\begin{theorem}\label{thm:antipode} For $\psi= (W,Q,\pi,I)$, using the notation as in Eq.~\eqref{TakeuchisFormula2} and Section~\ref{sec:comb}, we have
   $$S(\psi)=\sum_{{\mathcal K}\in \Psi^0({\mathcal FD})} (-1)^{\ell({\mathcal K})} a\big(\overline{G}({\mathcal K})\big)\cdot (W_{\mathcal K},Q_{\mathcal K},\pi_{\mathcal K},I_{\mathcal K})\,.$$
\end{theorem}

\begin{proof} In Equation~\eqref{TakeuchisFormula2}, we partition the set ${\mathcal FD}$ according to the values of $\Psi^0$. For ${\mathcal K}=(K_1,K_2,\ldots,K_\ell)\in \Psi^0({\mathcal FD})$ and
let ${\mathcal FD}_{\mathcal K} = \big\{ {\mathcal F}\in {\mathcal FD} : \Psi^0({\mathcal F})={\mathcal K} \big\}$. For any ${\mathcal F}\in{\mathcal FD}_{\mathcal K}$, we have that ${\mathcal F}\cong {\mathcal K}$. Hence Equation~\eqref{TakeuchisFormula2} can be written as
 $$ S(\psi)=\sum_{{\mathcal F}\in {\mathcal FD}} (-1)^{\ell({\mathcal F})}  (W_{\mathcal F},Q_{\mathcal F},\pi_{\mathcal F},I_{\mathcal F})=
      \sum_{{\mathcal K}\in \Psi^0({\mathcal FD})} \Big( \sum_{{\mathcal F}\in {\mathcal FD}_{\mathcal K}} (-1)^{\ell({\mathcal F})} \Big) (W_{\mathcal K},Q_{\mathcal K},\pi_{\mathcal K},I_{\mathcal K}) \,.
 $$
We now concentrate on the coefficient of $ (W_{\mathcal K},Q_{\mathcal K},\pi_{\mathcal K},I_{\mathcal K}) $ in the above sum. The theorem will follow as soon as we show that
 \begin{equation}\label{FixK}
\sum_{{\mathcal F}\in {\mathcal FD}_{\mathcal K}} (-1)^{\ell({\mathcal F})} =(-1)^{\ell({\mathcal K})} a\big(\overline{G}({\mathcal F})\big)\,.
 \end{equation}
We now partition ${\mathcal FD}_{\mathcal K}$. For all ${\mathcal F}=(F_1,F_2,\ldots,F_m)\in {\mathcal FD}_{\mathcal K}$, it is clear that $\overline{G}({\mathcal F})=\overline{G}({\mathcal K})$ since the simple graph depends only on $\Psi^0({\mathcal F})=K$. For any acyclic orientation $G$ of $\overline{G}({\mathcal K})$, we let ${\mathcal FD}_{\mathcal K}^G=\big\{  {\mathcal F}\in {\mathcal FD}_{\mathcal K} : G({\mathcal F})=G\big\}$. We have
 $$  \sum_{{\mathcal F}\in {\mathcal FD}_{\mathcal K}} (-1)^{\ell({\mathcal F})} = \sum_G \Big( \sum_{{\mathcal F}\in {\mathcal FD}_{\mathcal K}^G} (-1)^{\ell({\mathcal F})} \Big)
  $$
where the sum is over all acyclic orientations $G$ of $\overline{G}({\mathcal K})$. Equation~\eqref{FixK} follows as soon as we show
 \begin{equation}\label{FixKG}
\sum_{{\mathcal F}\in {\mathcal FD}_{\mathcal K}^G} (-1)^{\ell({\mathcal F})} =(-1)^{\ell({\mathcal K})} \,.
 \end{equation}
To prove this identity we construct an involution $\Phi_{\mathcal K}^G\colon {\mathcal FD}_{\mathcal K}^G\to {\mathcal FD}_{\mathcal K}^G$ such that
\begin{enumerate}
\item[(A)] $\Phi_{\mathcal K}^G({\mathcal F})={\mathcal F}$ if and only if $f_{\mathcal F}=\sigma_{\mathcal F}^{-1}$
\item[(B)] if $f_{\mathcal F}\ne \sigma_{\mathcal F}^{-1}$, then $\Phi_{\mathcal K}^G({\mathcal F})={\mathcal F}'\ne {\mathcal F}$ and $\ell({\mathcal F}')=\ell({\mathcal F})\pm 1$.
\end{enumerate}
When $f_{\mathcal F}=\sigma_{\mathcal F}^{-1}$, there is a unique ${\mathcal F}$ with this property and we must have $\ell({\mathcal F})=\ell({\mathcal K})=\ell$.
This is the unique fix point of $\Phi_{\mathcal K}^G$ and the sign agree with Equation~\eqref{FixKG}.

Now assume $f_{\mathcal F}\ne \sigma_{\mathcal F}^{-1}$. Find the smallest $i$ such that $f^{-1}_{\mathcal F}(i)\ne\{\sigma_{\mathcal F}(i)\}$. For $1\le r\le i-1$, we have $f^{-1}_{\mathcal F}(r)=\{\sigma_{\mathcal F}(r)\}$. As in the definition of $\sigma_{\mathcal F}$, let $V_i=\{1,2,\ldots,\ell\}\setminus \{\sigma_{\mathcal F}(1),\ldots,\sigma_{\mathcal F}(i-1)\}$.
We consider $G_i=G\big|_{V_i}$ to be the oriented subgraph of $G$ restricted to the vertices $V_i$.
All the elements in $f^{-1}_{\mathcal F}(i)$ are sources in the graph $G_i$. The value of $\sigma_{\mathcal F}(i)$ is the largest source in $G_i$. Since $f^{-1}_{\mathcal F}(i)\ne\{\sigma_{\mathcal F}(i)\}$, there must be a source in $f^{-1}_{\mathcal F}(i)$ with value strictly less than $\sigma_{\mathcal F}(i)$. Let $a=\min f^{-1}_{\mathcal F}(i)<\sigma_{\mathcal F}(i)$. We then find the smallest $j\ge i$ such that $f^{-1}_{\mathcal F}(j)$ contains a source of $G_i$ with value $b>a$. We let 
  $$U=\big\{ a'\in f^{-1}_{\mathcal F}(j) : \exists x\le a \text{ a source of } G_i   \text{ and a path from $x$ to $a'$} \big\}$$
If $U=\emptyset$, then $j>i$ since $a\not\in U$. In this case we remark that our choice of $j$ implies that all the element of $f^{-1}_{\mathcal F}(j)$ are connected to a source $x\le a$ and there is not edge in $G$ from any element of $f^{-1}_{\mathcal F}(j-1)$ to any element of $f^{-1}_{\mathcal F}(j)$. If $U=\emptyset$, then we define
\begin{equation}\label{a_merge}
 \Phi_{\mathcal K}^G({\mathcal F})={\mathcal F}'= (F_1,\ldots,F_{j-2},F_{j-1}\cup F_{j},F_{j+1},\ldots,F_m).
 \end{equation}
 Remark again that all the components of $F_{j-1}$ are orthogonal to all components of $F_j$ and thus ${\mathcal F'}<{\mathcal F}$. Moreover $\ell({\mathcal F'})=\ell({\mathcal F})-1$ and
  $$f^{-1}_{\mathcal F'}(r)= \begin{cases}
              f^{-1}_{\mathcal F}(r) &\text{ if $r<j-1$,}\cr 
              f^{-1}_{\mathcal F}(j-1)\cup f^{-1}_{\mathcal F}(j)&\text{ if $r=j-1$,}\cr 
              f^{-1}_{\mathcal F}(r+1)&\text{ if $r>j-1$.}\cr
              \end{cases}$$
It is easy to check that if we repeat the procedure above for ${\mathcal F}'$ we will obtain $i',a',j',U'$ in such a way that $i'=i$, $a'=a$, $j'=j-1$ and $U'=f^{-1}_{\mathcal F}(j-1)\ne\emptyset$.

Now we consider the case when $U\ne\emptyset$. Here we reverse the procedure just above. That is, let $U^c=f^{-1}_{\mathcal F}(j)\setminus U$. All the parts of $U$ are connected to a source of $G_i$ with value $\le a$. Let $F_U$ be the component of ${\mathcal K}$ in $F_j$ indexed by $U$, and let $F_{U^c}$ be the component of ${\mathcal K}$ in $F_j$ indexed by $U^c$. We have that $F_j=F_{U}\cup F_{U^c}$ and all elements of $F_{U}$ are orthogonal to all elements of $F_{U^c}$.
For $U\ne\emptyset$, we define
\begin{equation}\label{a_split}
 \Phi_{\mathcal K}^G({\mathcal F})={\mathcal F}'= (F_1,\ldots,F_{j-1},F_U,F_{U^c},F_{j+1},\ldots,F_m).
 \end{equation}
 Remark that now ${\mathcal F}<{\mathcal F'}$ and $\ell({\mathcal F'})=\ell({\mathcal F})+1$. Moreover
  $$f^{-1}_{\mathcal F'}(r)= \begin{cases}
              f^{-1}_{\mathcal F}(r) &\text{ if $r<j-1$,}\cr 
              U&\text{ if $r=j$,}\cr 
              U^c&\text{ if $r=j+1$,}\cr 
              f^{-1}_{\mathcal F}(r-1)&\text{ if $r>j+1$.}\cr
              \end{cases}$$
For this ${\mathcal F}'$ we will obtain $i',a',j',U'$ in such a way that $i'=i$, $a'=a$, $j'=j+1$ and $U'=\emptyset$. The map $ \Phi_{\mathcal K}^G$ is thus the desired involution.
\end{proof}

\begin{remark}\label{rem:anti}
The formula in Theorem~\ref{thm:antipode} is cancelation free. Only maximal refinement contribute to the formula and Lemma~\ref{lemma:lengthFmax} guaranties that isomorphic  refinement all have the same length. But the formula may contain the same basis element more than once. For example, the second and third terms in the antipode formula in Figure~\ref{fig:A3k2_antipode} are equivalent tuples of subword complexes and therefore represent the same basis element. In fact, the corresponding Coxeter groups are isomorphic ($A_2\times A_1$) and the words are the same up to commutation of consecutive commuting letters. The element $\pi$ and the facet is also the same in both terms. 
  
\begin{figure}[htb]
\begin{center}
\includegraphics[width=0.95\textwidth]{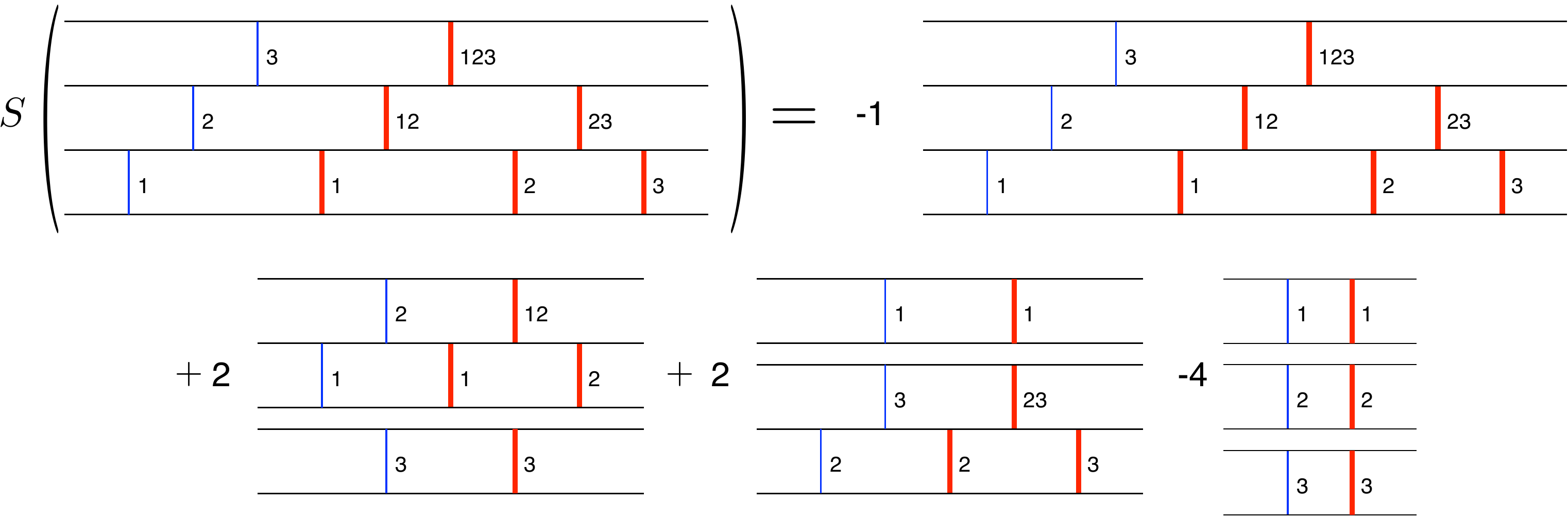}
\caption{Another example of the cancellation free formula for the antipode. Note that the second and third terms in the formula represent the same basis element, and the formula can be further simplified.}
\label{fig:A3k2_antipode}
\end{center}
\end{figure}  
\end{remark}

\begin{remark}
At this point it is an open problem to describe the space of primitive elements, its generators and its Lie structure. The main problem is our poor understanding 
of the equivalent classes. The problem exhibited in Remark~\ref{rem:anti} makes it very difficult to solve the equations
  $$\Delta_{\alpha}(X)=0$$
  for all composition $\alpha$ such that $\ell(\alpha)=2$. 
\end{remark}

\begin{remark}
In Section~\ref{sec:clusters} we introduce some interesting Hopf subalgebra of $\YY$.  But much more can be done in the study of the structure of the Hopf algebra of $\YY$.
In particular what are the characters on $\YY$ that give rise to interesting combinatorial Hopf algebras as in \cite{AguiarBergeronSottile}?
What are the associated even-odd sub-Hopf algebras? Is there new identities one can derive from this as in~\cite{AguiarHsiao}?
\end{remark}

%
%
%
%

\section{Hopf algebra of $c$-clusters of finite type}
\label{sec:clusters}

The Hopf algebra of subword complexes induces interesting sub-Hopf algebras on 

\begin{enumerate}[-]
\item subword complexes of finite type,
\item subword complexes of Cartesian products of type $A$,
\item root independent subword complexes, and
\item $c$-clusters of finite type. 
\end{enumerate}

The first two are clearly sub-Hopf algebras. The root independent subword complexes form an interesting family of examples that were introduced in the study of brick polytopes of spherical subword complexes in~\cite{PilaudStump-brickPolytopes}. They are subword complexes such that all the roots in the root configuration $\Roots{I}$ of a facet $I$ are linearly independent.
Pilaud and Stump~\cite{PilaudStump-brickPolytopes} show that the boundary of the brick polytope of a root independent subword complex is isomorphic to the dual of the subword complex, and use it to recover the polytopal constructions of generalized associahedra of Hohlweg, Lange and Thomas in~\cite{HohlwegLangeThomas}. The root independent subword complexes also have interesting connections with Bott-Samelson varieties and symplectic geometry, which have been used to describe the toric varieties of the associated brick polytopes~\cite{escobar_brick_2014}.   

The subword complex multiplication and comultiplication are clearly closed on the vector space generated by root independent subword complexes, which makes it into a graded sub-Hopf algebra. 
A remarkable subfamily of this family are the subword complexes associated to $c$-cluster complexes studied in~\cite{ceballos_subword_2013}. These complexes encode the combinatorics of the mutation graph of cluster algebras of finite types obtained from acyclic cluster seeds, and have been extensively studied and used in the literature~\cite{Reading-cambrianLattices,ReadingSpeyer,HohlwegLangeThomas,stella_polyhedral_2011,PilaudStump-brickPolytopes}. The Hopf algebra of subword complexes interestingly induces a non-trivial sub-Hopf algebra on $c$-clusters, and the rest of this section is devoted to its study.

\subsection{Hopf algebra of $c$-clusters of finite type}
In~\cite{ceballos_subword_2013}, Ceballos, Labb\'e and Stump showed that the $c$-cluster complexes arising from the the theory of cluster algebras can be obtained as well chosen subword complexes. More precisely, the $c$-cluster complex is the subword complex associated to the word~$Q=c\wo(c)$ and the longest element $\pi=\wo$, where $\wo(c)$ is the first lexicographically subword of $c^\infty$ that is a reduced expression of $\wo$.

\begin{theorem}[{\cite[Theorem 2.2]{ceballos_subword_2013}}]
For any finite Coxeter group, the subword complex $\subwordComplex{c\wo(c)}{\wo}$ is isomorphic to the $c$-cluster complex.
\end{theorem}

We will see below that the Hopf algebra of subword complexes induces a sub-Hopf algebra structure on this family of subword complexes. Let $\C_n$ be the subfamily of $Y_n$ corresponding to subword complexes of the form~$\subwordComplex{c\wo(c)}{\wo}$. 

\begin{theorem}\label{thm:subHopf_cwo}
The graded vector space
\[
\C := \bigoplus_{n\geq 0}  \ k[\C_n]
\]
equipped with the subword complex multiplication and comultiplication is a connected graded sub-Hopf algebra of the Hopf algebra of subword complexes.
\end{theorem}

As a consequence, we obtain.  

\begin{corollary}
The subword complex multiplication and comultiplication induce a graded Hopf algebra structure on the vector space generated by $c$-clusters of finite type. 
\end{corollary}

Before proving these results we briefly recall the definition of $c$-clusters and describe their Hopf algebra structure in the classical types.

\begin{figure}[tbp]
\includegraphics[width=0.95\textwidth]{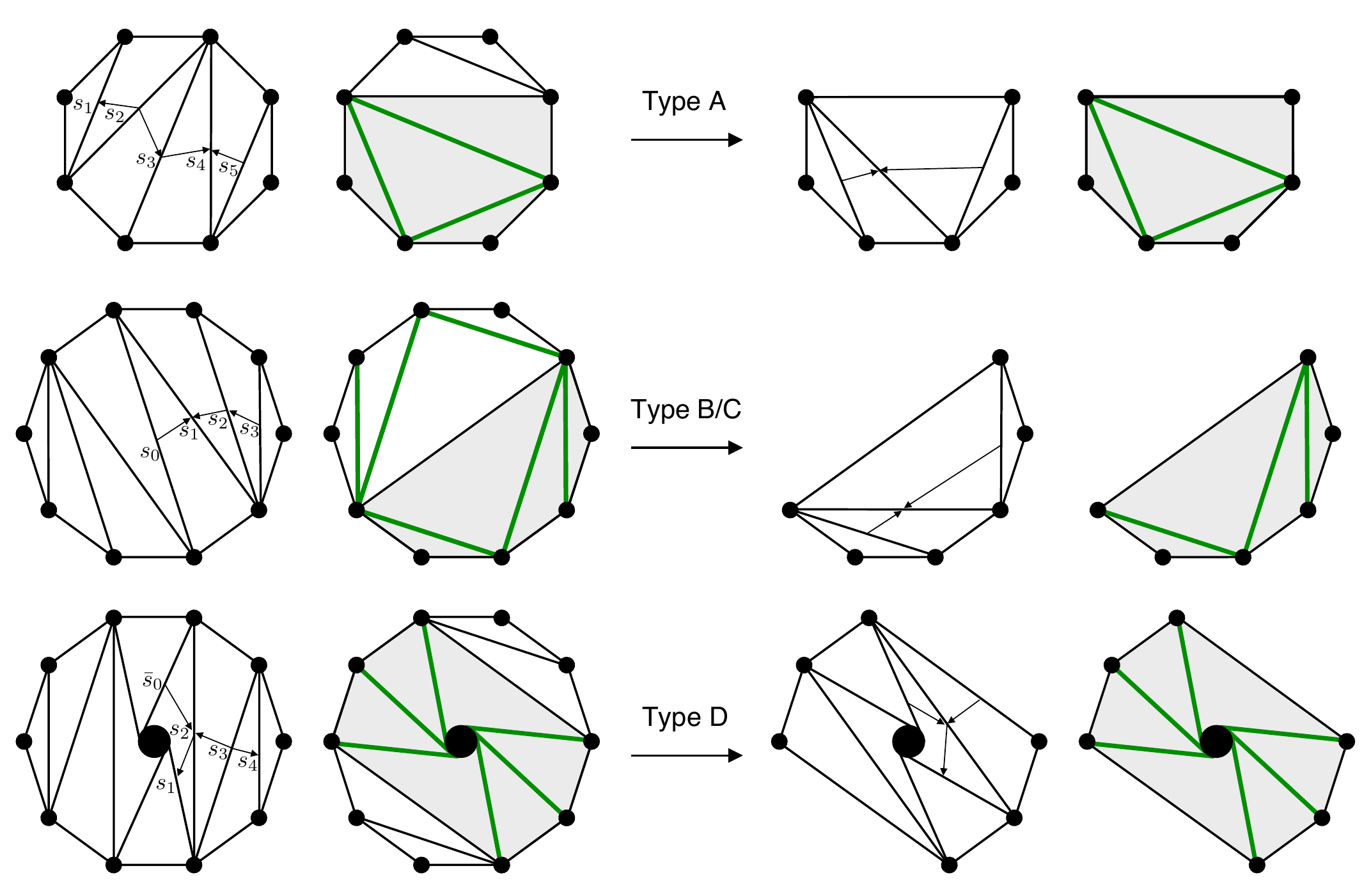}
\caption{Restriction of $c$-clusters in the classical types.}
\label{fig:c-clusters_classical}
\end{figure}

\begin{figure}[htbp]
\includegraphics[width=0.9\textwidth]{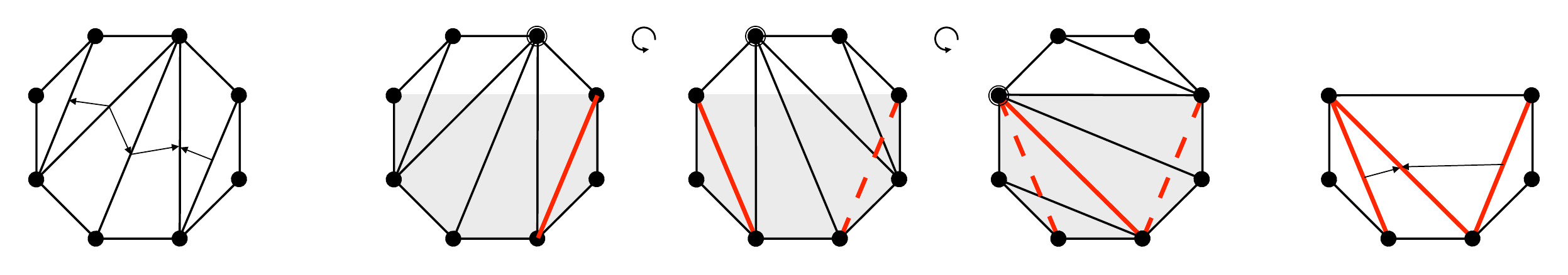}
\caption{Restriction of the acyclic cluster seed $c$ obtained by rotation in type~$A$. The other types are similar.}
\label{fig:c-clusters_classical_rotation}
\end{figure}

\subsection{$c$-clusters}
\label{sec:c-clusters}

Let~$W$ be a (non necessarily irreducible) finite Coxeter group and~$\Phi$ be an associated root system. The~$c$-cluster complex is a simplicial complex on the set of almost positive roots of~$\Phi$, which was introduced by Reading~\cite{Reading-coxeterSortable} following ideas from~\cite{MarshReinekeZelevinsky}. This complex generalizes the cluster complex of Fomin and Zelevinsky~\cite{FominZelevinsky-YSystems}, and has an extra parameter~$c$ corresponding to a Coxeter element. 

Given an acyclic cluster seed~$\seed$, the denominators of the cluster variables with respect to~$\seed$ are in bijection with the set of almost positive roots. The variables in~$\seed$ correspond to the negative roots, and any other variable to the positive root determined by the exponent of its denominator. 
Denote by~$\seed^c$ an acyclic cluster seed corresponding to a Coxeter element~$c$. It associated (weighted) quiver corresponds to the Coxeter graph oriented according to~$c$: a pair~$s,t\in S$ of non-commuting generators has the orientation~$s\rightarrow t$ if and only if~$s$ comes before~$t$ in~$c$.
The~$c$-clusters are the sets of almost positive roots corresponding to clusters obtained by mutations from~$\seed^c$. These were described in purely combinatorial terms using a notion of~$c$-compatibility relation by Reading in~~\cite{Reading-coxeterSortable}, and can be described purely in terms of the combinatorial models in the classical types~\cite[Section~5.4]{ceballos7}~\cite[Section~7]{CeballosPilaud}.

For the purpose of this paper, it is more convenient to consider~$c$-clusters as pairs~$(\seed^c,\cluster)$, where~$\seed^c$ is an acyclic cluster seed corresponding to~$c$ and~$\cluster$ is any cluster obtained from~$\seed^c$ by mutations. Note that this convention is more general than the original one. For example, two pairs related by rotation give the same $c$-cluster when considered as a set of almost positive roots. Figure~\ref{fig:c-clusters_classical} shows examples of $c$-cluster pairs $(\seed^c,\cluster)$ for the classical types. The acyclic (weighted) quiver associated to $\seed^c$ has nodes corresponding to the ``diagonals" of $\seed^c$ and directed arcs connecting clockwise consecutive (internal) sides of the ``triangles". The corresponding Coxeter elements $c$ and expressions $c\wo(c)$ for the acyclic seeds in Figure~\ref{fig:c-clusters_classical} are:
 \[
\begin{array}{rll}
\text{Type }A:  &  c= s_2s_1s_3s_5s_4, &  c\wo(c) = s_2s_1s_3s_5s_4 | s_2s_1s_3s_5s_4 | s_2s_1s_3s_5s_4 | s_2s_1s_3s_5 | s_2  \\
\text{Type }B/C: &  c= s_0s_3s_2s_1, & c\wo(c) =  s_0s_3s_2s_1 | s_0s_3s_2s_1 | s_0s_3s_2s_1 | s_0s_3s_2s_1 | s_0s_3s_2s_1     \\
 \text{Type }D:  &  c = \bar s_0s_3s_2s_1s_4,   & c\wo(c) = \bar s_0s_3s_2s_1s_4 | \bar s_0s_3s_2s_1s_4 | \bar s_0s_3s_2s_1s_4 | \bar s_0s_3s_2s_1s_4 | \bar s_0s_3s_2s_4 | \bar s_0 
\end{array}
\]

 The bijection in~\cite{ceballos_subword_2013} relating positions in $Q_c=c\wo(c)$ to cluster variables and facets of~$\subwordComplex{c\wo(c)}{\wo}$ to $c$-clusters, maps the position of $c_i$ in the prefix $c$ of $Q_c$ to the diagonal of~$\seed^c$ corresponding to~$c_i$. The image of any other position is determined by rotation.  
The rotation in the classical types correspond to rotating the polygon in counterclockwise direction, with the special rule in type $D$ of exchanging the central chords going to the left and to the right of the central disk after rotating~\cite{ceballos_cluster_2015}. The rotation in the word $Q_c$ maps the position of a letter $s$ in $Q_c$ to the position of the next occurrence of $s$ in $Q_c$, if possible, and to the first occurrence of $\wo s \wo$ otherwise. This bijection is illustrated for the example of type~$A$ in Figure~\ref{fig:c-clusters_classical_typeA}. For example, the first appearance of $s_1$ corresponds to the diagonal 13, whose node is labeled by $s_1$. The next three appearances of $s_1$ correspond to the diagonals 24, 35, 46, obtained by rotating 13 one step at a time in counterclockwise direction. Since there is no more $s_1$'s, the next rotation goes to the first appearance of $s_5=\wo s_1 \wo$, which corresponds to the diagonal 57.

\subsection{Hopf structure in the classical types}

Let $(\seed,\cluster)$ be a $c$-cluster pair consisting of an acyclic cluster seed $\seed$ and any cluster $\cluster$ obtained from $\seed$ by mutations. The multiplication of $c$-clusters is given by disjoint union, and the comultiplication is:
\[
\displaystyle
\Delta\left((\seed,\cluster)\right) :=  \sum_{U \subset \cluster} (\seed_U,\cluster_U) \otimes (\seed_{\cluster\smallsetminus U}, \cluster_{\cluster\smallsetminus U})
\]
where $\seed_U$ and $\cluster_U$  denote the restrictions of $\seed$ and $\cluster$ to $U$. The \defn{restriction} $\cluster_U$ of $\cluster$ is simply equal to the restricted cluster $U$. The \defn{restriction} $\seed_U$ of $\seed$ is not as clear since $U$ is not a subset of $\seed$, and turns out to be much more interesting. Denote by $\overline U$ the \defn{closure set} of all cluster variables that can be obtained by mutating elements of $U$ in the cluster $\cluster$ (excluding those in $\cluster\smallsetminus U$). In the examples of classical types in Figure~\ref{fig:c-clusters_classical}, the subset $U$ is represented by the thick diagonals and the closure~$\overline U$ is the set of diagonals that fit in the shaded regions. Here, by diagonals we mean, diagonals in type~$A$, centrally symmetric pairs of diagonals in types $B/C$, and centrally symmetric pairs of chords in type $D$. We refer to~\cite[Section~3.5]{FominZelevinsky-YSystems}\cite[Section~12.4]{FominZelevinsky-ClusterAlgebrasII} for the description of the geometric models for cluster algebras of classical types $A$ and $B/C$, and to~\cite{ceballos_cluster_2015} for type~$D$. 

To obtain $\seed_U$ from $\seed$ we proceed with the following rotation process: 
first take all the elements of $\seed$ that belong to the closure $\overline U$. Then, we consecutively rotate $\seed$ and take all its elements that belong to $\overline U$ and are compatible with all previously taken elements. The process finishes when we have taken as many elements as the cardinality of $U$.  
Figure~\ref{fig:c-clusters_classical} illustrates the restrictions of~$\seed$ and~$\cluster$ for some examples in the classical types. The restriction of the acyclic seed of type $A$ by rotation is explicitly illustrated in Figure~\ref{fig:c-clusters_classical_rotation}.  

\begin{remark}
The multiplication and comultiplication are closed for Cartesian products of type~$A$, but are not for types~$B/C$ and~$D$. For instance, the restriction of a $c$-cluster of type~$B/C$ may be a disjoint union of $c$-clusters of types~$A$ and~$B/C$. 
It is also interesting to see that two different diagonals in $\seed_U$ may come from the same diagonal in $\seed$. This happens in the example of type~$B/C$ in Figure~\ref{fig:c-clusters_classical}: the two diagonals corresponding to the end points of the quiver of $\seed_U$ come from the same centrally symmetric pair of diagonals of length 1 in $\seed$.
\end{remark}

\begin{remark}
Theorem~\ref{thm:subHopf_cwo} (and more specifically Proposition~\ref{prop:cluster_flat} below), guaranties that the rotation process for restriction of acyclic cluster seeds indeed finishes, and moreover, that it produces a cluster that is acyclic (corresponding to a Coxeter element $c_F$). Indeed, rotating $\seed$ to obtain $\seed_U$ is equivalent to take the first positions in $Q_c$ (up to commutation of consecutive commuting letters) whose corresponding root function evaluation belong to a given flat~$F$. The result $\seed_U$ corresponds to $c_F$ which is guarantied to be an acyclic seed by Proposition~\ref{prop:cluster_flat}.
This non-trivial fact is not valid if the cluster seed is not acyclic. For an example take the triangulations of a convex~6-gon given by $\seed=\{13,15,35\}$ and $\cluster=\{13,14,46\}$, and let $U=\{14\}\subset \cluster$. Rotating the seed $\seed$ never produces a diagonal in the closure $\overline U$ and the rotating process never finishes. It would be interesting to investigate to which extent the results in this section generalize for other (non-finite) cluster algebras obtained from acyclic cluster seeds.  
\end{remark} 

We have explicitly computed the restriction of a $c$-cluster of type $A$ in terms of flat decompositions of corresponding subword complex in Figure~\ref{fig:c-clusters_classical_typeA}. Note that the restricted $c$-cluster is exactly the $c$-cluster corresponding to the restricted subword complex. 

\begin{figure}[t]
\includegraphics[width=0.95\textwidth]{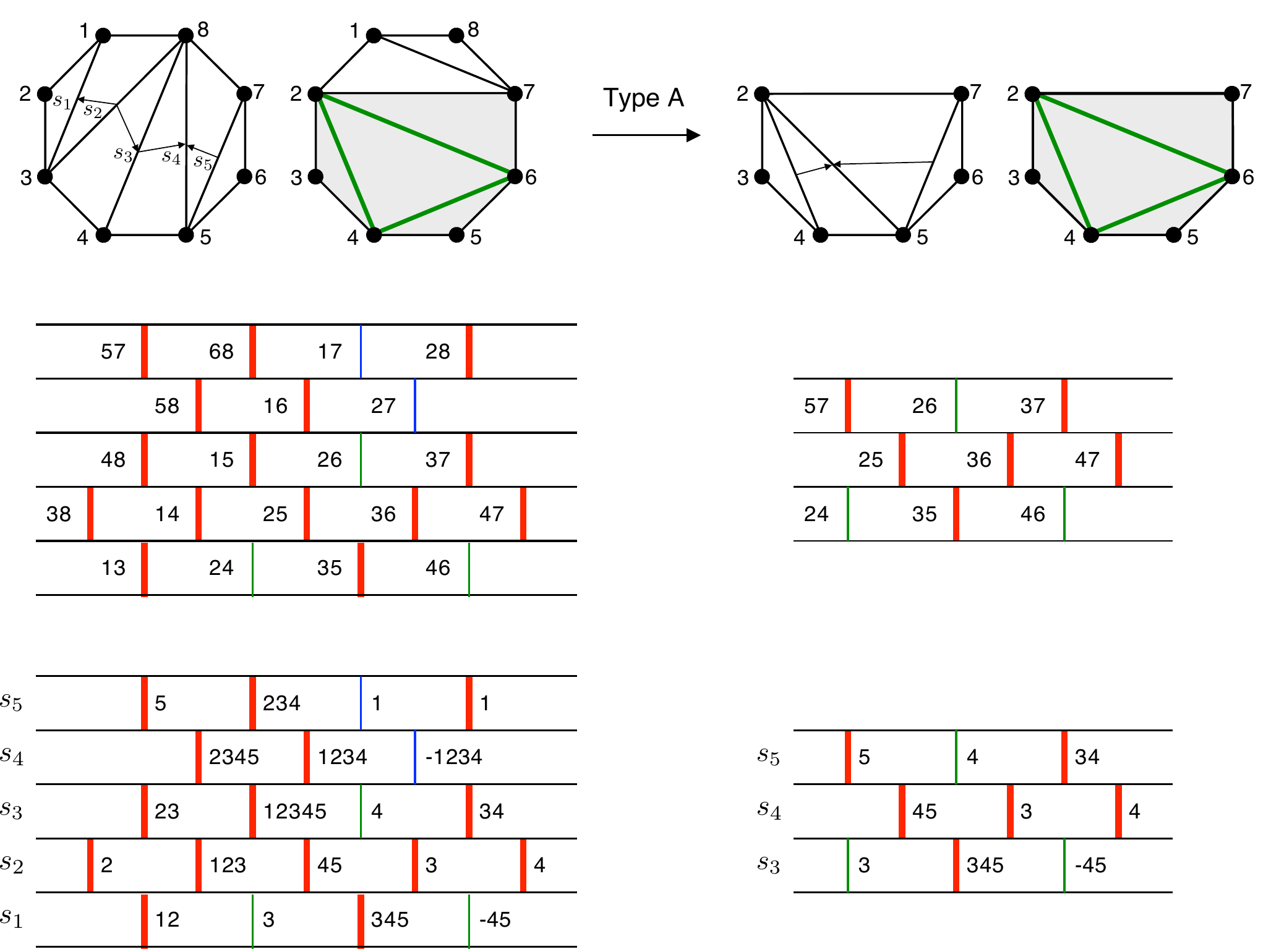}
\caption{Restriction of $c$-clusters via flat decompositions in the subword complex. (top) A $c$-cluster of type $A$ and a restriction of it. (middle) The corresponding subword complexes illustrating the bijection between diagonals of the polygon and letters of the word $c\wo(c)$. (bottom) The root function associated to the $c$-cluster facet and the flat decomposition of the restriction.}
\label{fig:c-clusters_classical_typeA}
\end{figure}

%
%

\subsection{Proof of Theorem~\ref{thm:subHopf_cwo}.} 
We need to show that the subword complex multiplication and comultiplication are closed in $\C$, the vector subspace corresponding to subword complexes of the form~$\subwordComplex{c\wo(c)}{\wo}$. The multiplication is clearly closed: the product of $\wo \in W$ and $\wo'\in W'$ is the longest element~$\overline w_\circ$ in the augmented Coxeter group~$\overline W$, and the concatenation of~$c\wo(c)$ with~$c'\wo(c')$ is equal to $cc'\wo(cc')$ up to commutation of consecutive commuting letters. The rest of this section is dedicated to prove that the comultiplication is closed. This fact follows from the following proposition and the fact that $\pi_F$ is the longest element in $W_F$ when $\pi=\wo$.

\begin{proposition}\label{prop:cluster_flat}
Let~$Q=c\wo(c)$ and $I$ be a facet of $\subwordComplex{Q}{\wo}$. If $F$ is a flat of $\Root{I}{Q}$ such that the span of $F$ is equal to the span of $R(I_F)$, then $Q_F=c_F\wo(c_F)$ up to commutation of consecutive commuting letters for a unique Coxeter element $c_F\in W_F$.  
\end{proposition}

Three particular cases of the proposition are proved in Lemmas~\ref{lem:prop_case1},~\ref{lem:prop_case2} and~\ref{lem:prop_case3}. The idea of the proof of the general statement is based on a notion of rotation of letters in the word~$Q$ and the understanding of flats after rotation.
Before showing these lemmas we need a characterization of the reduced expressions of $\wo$ that are equal to $\wo(c)$ up to commutation of consecutive commuting letters. This characterization follows from some results by Reading and Speyer~\cite{reading_sortable_2011} which we briefly summarize here.

\subsubsection{$c$-sorting words of sortable elements and the Reading--Speyer bilinear form $\omega_c$.} 


In~\cite{reading_sortable_2011}, Nathan Reading and David Speyer introduce an anti-symmetric bilinear form~$\omega_c$ indexed by a Coxeter element which is used to present a uniform approach to the theory of sorting words and sortable elements. The sortable elements of a Coxeter group play a fundamental role in the study of Cambrian fans. They are counted by the Coxeter Catalan numbers for groups of finite type and have interesting connections to noncrossing partitions and cluster algebras among others.  
An element $w\in W$ is called \emph{$c$-sortable} if the sequence of subsets of $[n]$ determined by the $c$-sorting word of $w$ in the ``blocks" of $c^\infty$ is weakly decreasing under inclusion. In this paper we do not use $c$-sortable elements in general but recall that the longest element $\wo$ is $c$-sortable for any Coxeter element $c$. The omega form $\omega_c(\beta,\beta')$ for a Coxeter element $c=s_1\dots s_n$ is determined by 
\[
\omega_c(\alpha_{s_i},\alpha_{s_j}) =     \left\{ \begin{array}{lcl}
         A(\alpha_{s_i},\alpha_{s_j}) & \mbox{for}
         & i>j, \\ 
	0  & \mbox{for} & i=j, \\
         -A(\alpha_{s_i},\alpha_{s_j}) & \mbox{for}
         & i<j,
                \end{array}\right.
\]
where $A$ is a symmetrizable Cartan matrix for $W$.  

\begin{lemma}[{\cite[Lemma~3.7]{reading_sortable_2011}}]
\label{lem:reading1}
Let $J\subset S$ and let $c'$ be the restriction of $c$ to the parabolic subgroup $W_J$. Then $\omega_c$ restricted to the the subspace $V_J\subset V$ is $\omega_{c'}$.
\end{lemma}

Let $a_1\dots a_k$ be a reduced expression for some $w\in W$. The \emph{reflection sequence} associated to $a_1\dots a_k$ is defined as the sequence $t_1,\dots , t_k$, where   $t_i=a_1a_2\dots a_i \dots a_2a_1$. The corresponding list of inversions is $\beta_1, \dots ,\beta_k$, where $\beta_i=a_1\dots a_{i-1} (\alpha_{a_i})=\beta_{t_i}\in \Phi^+$ is the unique positive root orthogonal to the reflection~$t_i$. 
\begin{lemma}[{\cite[Proposition~3.11]{reading_sortable_2011}}]
\label{lem:reading2}
Let $a_1 \dots a_k$ be a reduced word for some $w \in W$ with reflection sequence $t_1,\dots,t_k$. The following are equivalent:
\begin{enumerate}[(i)]
\item $\omega_c (\beta_{t_i},\beta_{t_j})\geq 0$ for all $i \leq j$ with strict inequality holding unless $t_i$ and $t_j$ commute.
\item $w$ is $c$-sortable and $(a_1,\dots, a_k)$ is equal to the $c$-sorting word for $w$ up to commutation of consecutive commuting letters.
\end{enumerate}
\end{lemma}

As a consequence of these two lemmas we obtain the following result.

\begin{lemma}\label{lem:parabolic_sorting}
Let $c$ be a Coxeter element and $c'$ be its restriction to a standard parabolic subgroup with root subsystem $\Phi'$. The restriction of the list of inversions of $\wo(c)$ to $\Phi'$ is equal to the list of inversions of a word equal to $\wo(c')$ up to commutation of consecutive commuting letters. 
\end{lemma}

\begin{proof}
Lemma~\ref{lem:reading2} characterizes all $c$-sorting words of sortable elements, up to commutation of consecutive commuting letters, in terms of the form $\omega_c$. This lemma, applied to the $c$-sorting word for $\wo$, lets you define an acyclic directed graph on the set of all reflections of the group, with an edge $t \rightarrow t'$ if and only if $t$ precedes $t'$ in the reflection sequence of the sorting word and $t$ and $t'$ do not commute.  
The set of words you can get from $\wo(c)$ by commutation of consecutive commuting letters is in bijection with the set of linear extensions of this directed graph.
By Lemma~\ref{lem:reading1}, the restriction of this graph to reflections in the parabolic subgroup satisfies exactly the same property in Lemma~\ref{lem:reading2}~$(i)$ relative to the parabolic. Therefore, the restricted list of reflections is the list of reflections of a word equal to $\wo(c')$ up to commutation of consecutive commuting letters.
\end{proof}

\subsubsection{Key lemmas and proof of Proposition~\ref{prop:cluster_flat}}

\begin{lemma}\label{lem:prop_case1}
Proposition~\ref{prop:cluster_flat} holds in the case when $I=\{1,\dots n\}$ is the facet of $\subwordComplex{c\wo(c)}{\wo}$ corresponding to the prefix $c$ in $c\wo(c)$.
\end{lemma}

\begin{proof}
Let $c=c_1\dots c_n$ and $c_F=c_{i_1}\dots c_{i_\ell}$ be the restriction of $c$ to the positions whose corresponding roots belong to the flat $F$. The subgroup $W_F$ is the parabolic subgroup generated by~$c_{i_1},\dots,c_{i_\ell}$, and the root subsystem $\Phi_F$ is the restriction of $\Phi$ to the span of $\alpha_{c_{i_1}},\dots,\alpha_{c_{i_\ell}}$. The flat $F$ is given by $(\alpha_{c_{i_1}},\dots,\alpha_{c_{i_\ell}})$ and the list of inversions of $\wo(c)$ that belong of $\Phi_F$. 
By Lemma~\ref{lem:parabolic_sorting}, this list of inversions is the list of inversions of a word $w$ equal to $\wo(c')$ up to commutation of consecutive commuting letters. The word $Q_F$ is exactly equal to $c'w$ which is equal to $c'\wo(c')$ up to commutation of consecutive commuting letters.
\end{proof}

\begin{lemma}\label{lem:prop_case2}
Proposition~\ref{prop:cluster_flat} holds in the case when $1\in I$ and $F$ is the codimension 1 flat of~$\Root{I}{Q}$ composed by the root vectors in the span of $R(I\smallsetminus 1)$.
\end{lemma}

\begin{proof}
The facet $I$ can be connected to the facet $I'=\{1,\dots, n\}$ by a sequence of flips involving positions with root vectors in the flat $F$. As in the proof of Theorem~\ref{thm:docompositon}, Lemma~\ref{lem:subword_flips} implies that these flips only depend on the root function, and therefore produce a sequence of flips in the restricted subword complex~$\subwordComplex{Q_F}{\wo}$. Since the word $Q_F$ is preserved under flips, it suffices to show the result for the facet $I'$ and the corresponding flat. This case follows from~Lemma~\ref{lem:prop_case1}.  
\end{proof}

\begin{lemma}\label{lem:prop_case3}
Proposition~\ref{prop:cluster_flat} holds in the case when $F$ is a codimension 1 flat of~$\Root{I}{Q}$.
\end{lemma}

The proof of this lemma uses a rotation of letters operation on subword complexes. The \defn{rotation of a word} $Q=(q_1,q_2 \dots, q_r)$ is the word $\rotated{Q}=(q_2, \dots, q_r, \woconj{q_1} )$, where $\woconj{q} := \wo q \wo$. 
 Using \cite[Proposition~3.9]{ceballos_subword_2013}, we see that $\subwordComplex{Q}{\wo}$ and its rotation $\subwordComplex{\rotated{Q}}{\wo}$ are isomorphic. The isomorphism sends a position $i$ in $Q$ to the \defn{rotated position} $\rotated{i}$ in $\rotated{Q}$, which is by definition equal to $r$ if~$i= 1$ and equal to $i-1$ otherwise. Under this isomorphism, the facet $I$ is mapped to its \defn{rotated facet} $\rotated{I}=\{\rotated{i}: i\in I\}$. This rotation operation behaves very well in the family of subword complexes of the form~$\subwordComplex{c\wo(c)}{\wo}$. 

\begin{lemma}[{\cite[Proposition~4.3]{ceballos_subword_2013}}]
\label{lem:rotationCLS}
If $Q=c\wo(c)$ then the rotated word $\rotated{Q}=c'\wo(c')$ up to commutation of consecutive commuting letters, for a Coxeter element $c'$.
\end{lemma}

Given a flat $F$ of $\Root{I}{Q}$, one can also define the \defn{rotated flat} $\rotated{F}$ of $\Root{\rotated{I}}{\rotated{Q}}$ by:
\[
F=(\Root{I}{j_1},\dots, \Root{I}{j_{r'}}), \hspace{1cm} 
\rotated{F}= (\Root{\rotated{I}}{\rotated{{j_1}}},\dots, \Root{\rotated{I}}{\rotated{{j_{r'}}}}).
\] 
One can check that~$\rotated{F}$ is indeed a flat by seeing how the root function is transformed under rotation: if $1\in I$, the root at the first position becomes negative and is rotated to the end, while all other roots are preserved. if $1\notin I$, the first root is rotated to the end and all other roots~$\beta$ become $q_1(\beta)$. In the first case, being a flat is clearly preserved after rotating. In the second case, the rotated flat is the list of roots in $\Root{\rotated{I}}{\rotated{Q}}$ that belong to the subspace $q_1(V_F)$, which is clearly a flat. We also observe that the root subsystem $\Phi_{\rotated{F}}$ is exactly equal to $\Phi_F$ in the first case, and isomorphic to it in the second. In both cases we have $W_{\rotated{F}}\cong W_F$. Moreover, the words~$Q_F$ and~$Q_{\rotated{F}}$ are either equal to each other or are connected by a rotation (via this isomorphism). Indeed, if the first root of $Q$ does not belong to the flat $F$ then $Q_{\rotated{F}}= Q_F$ (after applying the isomorphism), otherwise $Q_{\rotated{F}}= \rotated{Q_F}$ (after applying the isomorphism). As a consequence we get the following lemma.

\begin{lemma}\label{lem:rotation_flat}
Let $I$ be a facet of a subword complex $\subwordComplex{Q}{\wo}$ and $F$ be a flat of~$\Root{I}{Q}$. If $Q'$ is a word obtained from $Q$ by a sequence of rotations and $I',F'$ are the corresponding rotated facet and flat, then
\begin{enumerate}[(i)]
\item $W_{F'} \cong W_F$,
\item $Q'_{F'}$ can be obtained from $Q_F$ by a sequence of rotations (via the isomorphism of the underlying Coxeter groups), and
\item $I'_{F'}$ is the corresponding rotation of the facet $I_F$.
\end{enumerate}
\end{lemma}

Now we are ready to prove Lemma~\ref{lem:prop_case3} and Proposition~\ref{prop:cluster_flat}.
\begin{proof}[Proof of Lemma~\ref{lem:prop_case3}]
Let~$Q=c\wo(c)$, $I$ be a facet of $\subwordComplex{Q}{\wo}$ and $F$ be a codimension 1 flat of~$\Root{I}{Q}$ such that the span of $F$ is equal to the span of $R(I_F)$. We need to show that $Q_F$ is equal to $c_F\wo(c_F)$ up to commutation of consecutive commuting letters for a unique Coxeter element~$c_F\in W_F$. 
Since all the roots in $R(I)$ are linearly independent, $I_F=I\smallsetminus i$ for some $i\in I$, and 
the flat $F$ is composed by the root vectors in the span of $R(I\smallsetminus i)$. Applying $i-1$ rotations we obtain a word $Q'$ where position $i$ is rotated to position 1, a rotated facet $I'$ and a rotated flat~$F'$. By Lemma~\ref{lem:rotationCLS}, $Q'=c'\wo(c')$ up to commutation of consecutive commuting letters for some Coxeter element $c'$. Moreover, $F'$ is the codimension 1 flat of~$\Root{I'}{Q'}$ composed by the root vectors in the span of~$R(I'\smallsetminus 1)$. Lemma~\ref{lem:prop_case2} then implies that $Q'_{F'}=c_{F'}\wo(c_{F'})$ up to commutation of consecutive commuting letters for some Coxeter element $c_{F'}$. Since~$Q_F$ and~$Q'_{F'}$ are connected by rotations via an isomorphism (Lemma~\ref{lem:rotation_flat}~$(ii)$), applying Lemma~\ref{lem:rotationCLS} again guaranties that~$Q_F=c_F\wo(c_F)$ up to commutation of consecutive commuting letters for some Coxeter element $c_F$. This Coxeter element is clearly unique.   
\end{proof}

\begin{proof}[Proof of Proposition~\ref{prop:cluster_flat}]
The restricted subword complex can be consecutively obtained by restricting to codimension 1 flats. The results then follows from Lemma~\ref{lem:prop_case3}.
\end{proof}

%
%
%
%

\appendix
\section{Geometric interpretation of the inversions of a word}\label{Appendix:inversions}

Let $(W,S)$ be a possibly infinite Coxeter system acting on a vector space $V$ generated by simple roots $\Delta$, and let~$\Phi=\Phi^+ \sqcup  \Phi^- \subset \R^n$ be a root system associated to it. 
 For a given (not necessarily reduced) word~$P=(p_1,\dots,p_r)$ in the generators $S$, the \defn{inversions} of~$P$ are the roots $\gamma_1,\dots,\gamma_r$ defined by 
\[
\gamma_i := p_1\dots p_{i-1}(\alpha_{p_i}).
\] 
The list $\inv(P):=(\gamma_1,\dots, \gamma_r)$ is called the \defn{list of inversions} of $P$. Note that if $P$ is reduced $\inv(P)$ consists of $r$ different positive roots, while if $P$ is not reduced $\inv(P)$ may contain negative roots as well as repetitions. 

In this appendix, we present a geometric interpretation of the list of inversions of~$P$ in terms of walks in the geometric presentation of the group. 
In order to keep the intuition from finite reflection groups we distinguish the two cases of finite and infinite Coxeter groups. We refer to~\cite{humphreys_reflection_1992} for a more detailed study of root systems and Coxeter groups.

\subsection{Finite Coxeter groups}
Let $\A$ be the hyperplane arrangement of all reflections induced by~$\Phi$. For each hyperplane~$H\in\A$ there is a unique positive root $\alpha_H\in \Phi^+$ orthogonal to it.
We let~$H^+=\{v\in\R^n : \langle v,\alpha_H\rangle>0 \}$ where $\langle -,-\rangle$ is the canonical scalar product on $\R^n$. Similarly, let $H^-=-H^+$. The triples~$H^-,H,H^+$ decompose $\R^n$ into two half spaces and a subspace on codimension 1. The \defn{Coxeter complex} of $W$ is a cell decomposition of $\R^n$ obtained by considering all possible non-empty intersections $\bigcap_{H\in\A} H^{\epsilon(H)}$ where $\epsilon(H)$ is either $+, -$ or empty.
The \defn{fundamental chamber} is the $n$-dimensional cell we obtain by choosing $\epsilon(H)=+$ for all $H\in A$. The chambers of the complex (the $n$-dimensional cells) are in natural bijection with the elements of $W$. The walls of the chambers (the codimension 1 cells of the complex) can be naturally labeled according to the action of the group on the walls of the fundamental chamber. Figure~\ref{fig:labellingA2} illustrates an example of a labelling for the Coxeter group $W=A_2$. Note that the labelling of the walls is not unique, however, we will provide a precise labelling below, which will be useful for the purposes of this appendix. We refer to~\cite[Section~1.15]{humphreys_reflection_1992} for more details about the Coxeter complex. 

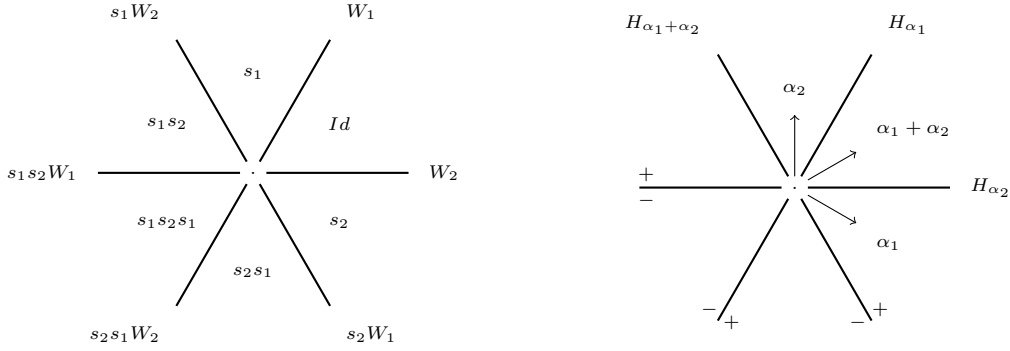
\begin{figure}[htbp]
 \begin{subfigure}{0.42\textwidth} 
 	{\begin{tikzpicture}[scale=.22,font=\tiny]
		\node (0) at (0:0) {.};
		\node (1) at (0:10) {};
		\node (2) at (60:10){};
		\node (3) at (120:10){};
		\node (4) at (180:10) {};
		\node (5) at (240:10) {};
		\node (6) at (300:10) {};
		
		\draw[thick] (0)--(1);
		\draw[thick] (0)--(2);
		\draw[thick] (0)--(3);
		\draw[thick] (0)--(4);
		\draw[thick] (0)--(5);
		\draw[thick] (0)--(6);
		
		\node (7) at (30:6) {$\small Id$};
		\node  (8) at (90:6) {$\small s_1$};
		\node (9) at (150:6) {$\small s_1s_2$};
		\node (10) at (210:6) {$\small s_1s_2s_1$};
		\node (11) at (270:6) {$\small s_2s_1$};
		\node (12) at (330:6) {$\small s_2$};

		\node[right] (13) at (1) {$\small W_{2}$};
		\node[above right] (14) at (2) {$\small W_{1}$};
		\node[above left] (15) at (3) {$\small s_1W_{2}$};
		\node[left] (16) at (4) {$\small s_1s_2W_{1}$};
		\node[below left] (17) at (5) {$\small s_2s_1W_{2}$};
		\node[below right] (18) at (6) {$\small s_2W_{1}$};	
		\end{tikzpicture}}
\end{subfigure}
\quad  
\begin{subfigure}{0.42\textwidth} \centering
 	{\begin{tikzpicture}[scale=.22,font=\tiny]
		\node (0) at (0:0) {.};
		\node (1) at (0:10) {};
		\node (2) at (60:10){};
		\node (3) at (120:10){};
		\node (4) at (180:10) {};
		\node (5) at (240:10) {};
		\node (6) at (300:10) {};
		\node (7) at (330:5) {};
		\node (8) at (90:5) {};
		\node (9) at (30:5) {};
		
		\draw[thick] (0)--(1);
		\draw[thick] (0)--(2);
		\draw[thick] (0)--(3);
		\draw[thick] (0)--(4);
		\draw[thick] (0)--(5);
		\draw[thick] (0)--(6);
		
		\draw[->] (0) to (7);
		\draw[->] (0) to (8);
		\draw[->] (0) to (9);

		\node[below right] at (7) {$\small \alpha_1$};
		\node[above] at (8) {$\small \alpha_2$};
		\node[above right] at (9) {$\small \alpha_1 + \alpha_2$};

		\node[right] at (1) {$\small H_{\alpha_2}$};
		\node[above right] at (2) {$\small H_{\alpha_1}$};
		\node[above left] at (3) {$\small H_{\alpha_1+\alpha_2}$};

		\node at (-185:9) {$+$};
		\node at (-175:9) {$-$};
		\node at (-125:9) {$-$};
		\node at (-115:9) {$+$};
		\node at (-65:9) {$-$};
		\node at (-55:9) {$+$};

	\end{tikzpicture}}
\end{subfigure}
\caption{(Left) Coxeter complex of type $A_2$. The chambers are labeled by elements of $W$, the walls are labeled according to the action of the group on the walls of the fundamental chamber. (Right) The hyperplanes are labeled with their unique orthogonal positive root, the positive and negative sides of each hyperplane are also shown.}
\label{fig:labellingA2}
\end{figure}

A word~$P=(p_1,\dots,p_r)$ in the generators of the group corresponds to a path from the fundamental chamber to the chamber corresponding to the element $p_1\dots p_r\in W$, crossing only through codimension~$\le 1$ cells. The $i$th wall crossed by the path is the wall with label $p_1\dots p_{i-1}W_{p_i}$, which is orthogonal to the inversion $\gamma_i=p_1\dots p_{i-1}(\alpha_{p_i})$. 
The inversion~$\gamma_i$ is a positive (resp. negative) root if the path crosses the $i$th wall from the positive (resp. negative) side of the hyperplane $H_{\gamma_i}$ to the negative (resp. positive).
Figure~\ref{fig:pathA2} illustrates an example for the Coxeter group $W=A_2$.
The description of the sign of $\gamma_i$ follows from ~\cite[Theorem in Section~5.4]{humphreys_reflection_1992}, which affirms that~$\ell(ws)>\ell(w)$ if and only if $w(\alpha_s)>0$, for $w\in W$ and $s\in S$.

\begin{figure}[htbp]
 \begin{subfigure}{0.47\textwidth} 
 	{\begin{tikzpicture}[scale=.22,font=\tiny,pathcolor/.style={color=blue!99!black}]
		\node (0) at (0:0) {.};
		\node (1) at (0:10) {};
		\node (2) at (60:10){};
		\node (3) at (120:10){};
		\node (4) at (180:10) {};
		\node (5) at (240:10) {};
		\node (6) at (300:10) {};
		
		\draw[thick] (0)--(1);
		\draw[thick] (0)--(2);
		\draw[thick] (0)--(3);
		\draw[thick] (0)--(4);
		\draw[thick] (0)--(5);
		\draw[thick] (0)--(6);
		
		\node[right] at (1) {$\small W_{2}=s_1s_2s_1s_2s_1W_{2}$};
		\node[above right] at (2) {$\small W_{1}=s_1s_2s_1s_2s_1s_2W_{1}$};
		\node[above left] at (3) {$\small s_1W_{2}$};
		\node[left] at (4) {$\small s_1s_2W_{1}$};
		\node[below left] at (5) {$\small s_1s_2s_1W_{2}$};
		\node[below right] at (6) {$\small s_1s_2s_1s_2W_{1}$};
		
		\draw [pathcolor,thick,domain=30:510,variable=\t,smooth,samples=75,densely dotted,-<]
        			plot ({\t }: 0.003*\t+4);
			
		\node at (-90:14) {$P=(s_1,s_2,s_1,s_2,s_1,s_2,s_1,s_2)$};	
		
	\end{tikzpicture}}
\end{subfigure}
\quad  
\begin{subfigure}{0.47\textwidth} \centering
 	{\begin{tikzpicture}[scale=.22,font=\tiny,pathcolor/.style={color=blue!99!black}]
		\node (0) at (0:0) {.};
		\node (1) at (0:10) {};
		\node (2) at (60:10){};
		\node (3) at (120:10){};
		\node (4) at (180:10) {};
		\node (5) at (240:10) {};
		\node (6) at (300:10) {};
		
		\draw[thick] (0)--(1);
		\draw[thick] (0)--(2);
		\draw[thick] (0)--(3);
		\draw[thick] (0)--(4);
		\draw[thick] (0)--(5);
		\draw[thick] (0)--(6);
		
		\node[right] (10) at (1) {$\small H_{\alpha_2}$};
		\node[above right] (11) at (2) {$\small H_{\alpha_1}$};
		\node[above left] (12) at (3) {$\small H_{\alpha_1+\alpha_2}$};

		\node at (-185:9) {$+$};
		\node at (-175:9) {$-$};
		\node at (-125:9) {$-$};
		\node at (-115:9) {$+$};
		\node at (-65:9) {$-$};
		\node at (-55:9) {$+$};
		
		\draw [pathcolor,thick,domain=30:510,variable=\t,smooth,samples=75,densely dotted,-<]
        			plot ({\t }: 0.003*\t+4);
			
		\node at (-90:14) {$\inv(P)=(\alpha_1,\alpha_1+\alpha_2,\alpha_2,-\alpha_1,-\alpha_1-\alpha_2,-\alpha_2,\alpha_1,\alpha_1+\alpha_2)$};	

	\end{tikzpicture}}
\end{subfigure}
\caption{Geometric interpretation of the inversions of a word.}
\label{fig:pathA2}
\end{figure}
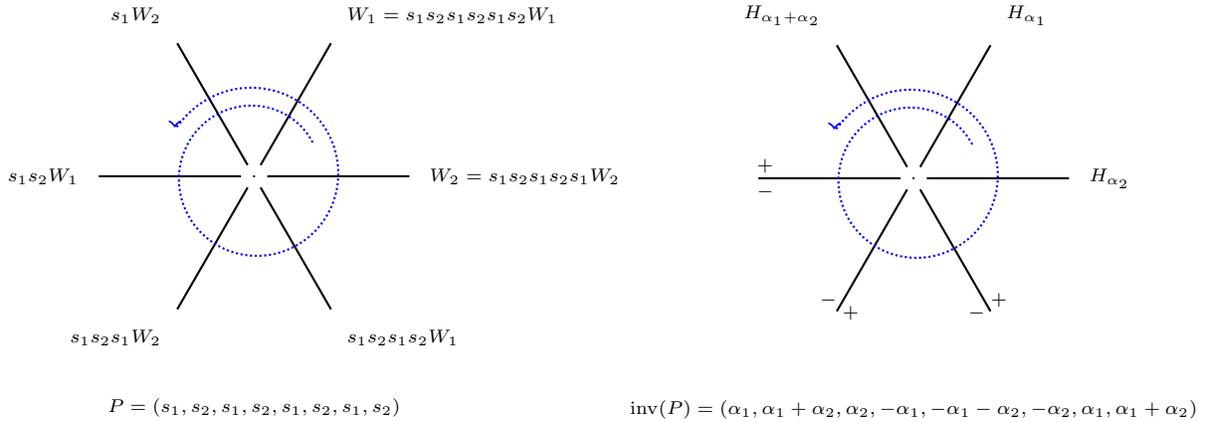

\subsection{Infinite Coxeter groups}
In the infinite case, 
everything works precisely the same way if we replace the Coxeter arrangement by the Tits cone and the orthogonal space to a root by its dual hyperplane. 
The Coxeter group $W$ is considered acting on the dual space $V^*$. 
For $f\in V^*$ and~$\lambda\in V$ we denote by $\langle f,\lambda \rangle$ the image of $\lambda$ under $f$. The action of $w\in W$ on $V^*$ is then characterized by   
\[
\langle w(f), w(\lambda) \rangle  = \langle f, \lambda \rangle. 
\]
The \defn{dual space} of a root $\alpha\in \Phi$ is the hyperplane~$Z_\alpha=\{f\in V^* : \langle f,\alpha \rangle=0 \}$, and the \defn{positive and negative dual half spaces} are~$Z_\alpha^+=\{f\in V^* : \langle f,\alpha \rangle>0 \}$ and~$Z_\alpha^-=\{f\in V^* : \langle f,\alpha \rangle<0 \}$ respectively. In particular, $Z_{w(\alpha)}=wZ_\alpha$. Let $C$ be the intersection of all $Z_{\alpha_s}^+$ for $s\in S$, and~$D:=\overline C$. The \defn{Tits cone} is the union of all $w(D)$ for $w\in W$.  
 
The cone $D$ can be naturally partitioned into subsets 
$C_I:=\left( \bigcap_{s\in I} Z_{\alpha_s} \right) \cap \left( \bigcap_{s\notin I} Z_{\alpha_s}^+ \right)$ for $I\subset S$. 
In particular, $C_\emptyset = C$ and $C_S=\{0\}$. 
This gives a natural cell decomposition of the Tits cone, whose cells can be naturally labeled according to the action of the group on the cells of $D$. The maximal cones correspond to elements of the group, with $C$ corresponding to the identity element.    
We refer to~\cite[Section~5.13]{humphreys_reflection_1992} for more details about the Tits cone.

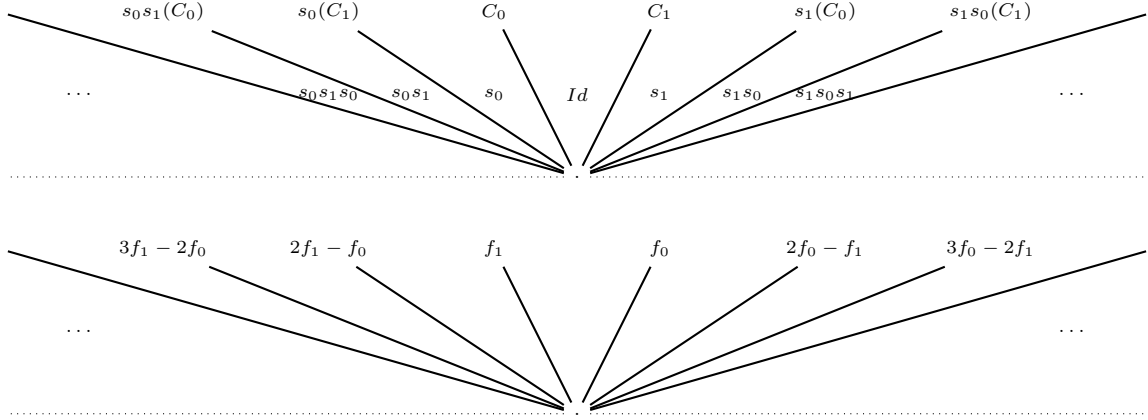
\begin{figure}[h]
\begin{subfigure}{1\textwidth} 
\centering
 	{\begin{tikzpicture}[scale=1.1,font=\tiny]
		\node (0) at (0,0) {.};
		\node (a1) at (1,2) {$C_1$};
		\node (a2) at (3,2) {$s_1(C_0)$};
		\node (a3) at (5,2) {$s_1s_0(C_1)$};
		\node (a4) at (7,2) {};
		
		\node (b1) at (-1,2) {$C_0$};
		\node (b2) at (-3,2){$s_0(C_1)$};
		\node (b3) at (-5,2){$s_0s_1(C_0)$};
		\node (b4) at (-7,2) {};

		\node (c1) at (-7,0) {};
		\node (c2) at (7,0) {};

		\node at (6,1) {$\dots$};		
		\node at (-6,1) {$\dots$};		
		
		\node at (0,1) {$Id$};
		\node at (1,1) {$s_1$};
		\node at (2,1) {$s_1s_0$};
		\node at (3,1) {$s_1s_0s_1$};

		\node at (-1,1) {$s_0$};
		\node at (-2,1) {$s_0s_1$};
		\node at (-3,1) {$s_0s_1s_0$};
		
		\draw[thick] (0)--(a1);
		\draw[thick] (0)--(a2);
		\draw[thick] (0)--(a3);
		\draw[thick] (0)--(a4);

		\draw[thick] (0)--(b1);
		\draw[thick] (0)--(b2);
		\draw[thick] (0)--(b3);
		\draw[thick] (0)--(b4);

		\draw[dotted] (0)--(c1);
		\draw[dotted] (0)--(c2);

		\end{tikzpicture}}
\end{subfigure}
\vspace{0.5cm}  

\begin{subfigure}{1\textwidth} 
\centering
 	{\begin{tikzpicture}[scale=1.1,font=\tiny]
		\node (0) at (0,0) {.};
		\node (a1) at (1,2) {$f_0$};
		\node (a2) at (3,2) {$2f_0-f_1$};
		\node (a3) at (5,2) {$3f_0-2f_1$};
		\node (a4) at (7,2) {};
		
		\node (b1) at (-1,2) {$f_1$};
		\node (b2) at (-3,2){$2f_1-f_0$};
		\node (b3) at (-5,2){$3f_1-2f_0$};
		\node (b4) at (-7,2) {};

		\node (c1) at (-7,0) {};
		\node (c2) at (7,0) {};

		\node at (6,1) {$\dots$};		
		\node at (-6,1) {$\dots$};		
		
		\draw[thick] (0)--(a1);
		\draw[thick] (0)--(a2);
		\draw[thick] (0)--(a3);
		\draw[thick] (0)--(a4);

		\draw[thick] (0)--(b1);
		\draw[thick] (0)--(b2);
		\draw[thick] (0)--(b3);
		\draw[thick] (0)--(b4);

		\draw[dotted] (0)--(c1);
		\draw[dotted] (0)--(c2);
		\end{tikzpicture}}
\end{subfigure}
\caption{(Top) The Tits cone of affine type $\tilde A_1$. The maximal cones are labeled by the elements of the group. The walls are labeled according to the action of the group on the walls of $C$. (Bottom) The rays labeled by linear combinations of the generators $f_0,f_1\in V^*$.}
\label{fig:affineA1_titscone}
\end{figure}

\begin{example}
Let $W=\tilde A_1$ be the affine Coxeter group generated by $s_0$ and $s_1$, with ${m(s_0,s_1)=\infty}$, and $\Phi \subset V$ be the associated root system with simple roots $\alpha_0$ and $\alpha_1$. The action of the group on~$V$ is determined by 
\[
\begin{array}{cccc}
 s_0(\alpha_0)=-\alpha_0, &  s_0(\alpha_1)=2\alpha_0+\alpha_1, &  s_1(\alpha_0)=\alpha_0+2\alpha_1, & s_1(\alpha_1)=-\alpha_1. 
\end{array}
\]
Let $f_0,f_1\in V^*$ defined by $\langle f_i, \alpha_j \rangle=\delta_{ij}$, which is equal to 1 if $i=j$ and to $0$ otherwise. The action of $W$ on $V^*$ is determined by 
\[
\begin{array}{cccc}
 s_0(f_0)=2f_1-f_0, &  s_0(f_1)=f_1, &  s_1(f_0)=f_0, & s_1(f_1)=2f_0-f_1. 
\end{array}
\]
For example $\langle s_0(f_0), \alpha_0 \rangle = \langle f_0, s_0(\alpha_0) \rangle=\langle f_0, -\alpha_0 \rangle=-1$ and 
$\langle s_0(f_0), \alpha_1 \rangle = \langle f_0, s_0(\alpha_1) \rangle=\langle f_0, 2\alpha_0+\alpha_1 \rangle=2$, which implies that $s_0(f_0)= 2f_1-f_0$. The cone $C$ is the set of positive linear combinations of $f_0$ and $f_1$, and the Tits cone of type $\tilde A_1$ is illustrated in Figure~\ref{fig:affineA1_titscone}. 
\end{example} 

Similarly as before, a word~$P=(p_1,\dots,p_r)$ in the generators of the group corresponds to a path in the Tits cone (instead of the Coxeter complex). This path goes from $C$ to the maximal cone corresponding to the element $p_1\dots p_r\in W$, and crosses only through codimension~${\le 1}$ cells. The~$i$th wall crossed by the path is the wall $p_1\dots p_{i-1}C_{p_i}$, which is contained in the dual hyperplane to the inversion~$\gamma_i=p_1\dots p_{i-1}(\alpha_{p_i})$. 
The inversion~$\gamma_i$ is a positive (resp. negative) root if the path crosses the $i$th wall from the positive (resp. negative) side of the hyperplane $Z_{\gamma_i}$ to the negative (resp. positive).
Figure~\ref{fig:affineA1_path} illustrates an example for the affine Coxeter group $W=\tilde A_1$.
The description of the sign of $\gamma_i$ follows from~\cite[Lemma in Section~5.13]{humphreys_reflection_1992}.

\begin{figure}[htbp]
\begin{subfigure}{1\textwidth} 
\centering
 	{\begin{tikzpicture}[scale=1.1,font=\tiny]
		\node (0) at (0,0) {.};
		\node (a1) at (1,2) {$C_1$};
		\node (a2) at (3,2) {$s_1(C_0)$};
		\node (a3) at (5,2) {$s_1s_0(C_1) = s_1s_0s_1(C_1)$};
		\node (a4) at (7,2) {};
		
		\node (b1) at (-1,2) {};
		\node (b2) at (-3,2){};
		\node (b3) at (-5,2){};
		\node (b4) at (-7,2) {};

		\node (c1) at (-7,0) {};
		\node (c2) at (7,0) {};

		\node at (6,1) {$\dots$};		
		\node at (-6,1) {$\dots$};		
		
		\node at (0,-0.5) {$P=(s_1,s_0,s_1,s_1)$};
				
		\draw[thick] (0)--(a1);
		\draw[thick] (0)--(a2);
		\draw[thick] (0)--(a3);
		\draw[thick] (0)--(a4);

		\draw[thick] (0)--(b1);
		\draw[thick] (0)--(b2);
		\draw[thick] (0)--(b3);
		\draw[thick] (0)--(b4);

		\draw[dotted] (0)--(c1);
		\draw[dotted] (0)--(c2);

		\end{tikzpicture}}
\end{subfigure}
\vspace{0.5cm}  

\begin{subfigure}{1\textwidth} 
\centering
 	{\begin{tikzpicture}[scale=1.1,font=\tiny,pathcolor/.style={color=blue!99!black}]
		\node (0) at (0,0) {.};
		\node (a1) at (1,2) {$Z_{\alpha_1}$};
		\node (a2) at (3,2) {$Z_{\alpha_0+2\alpha_1}$};
		\node (a3) at (5,2) {$Z_{2\alpha_0+3\alpha_1}$};
		\node (a4) at (7,2) {};
		
		\node (b1) at (-1,2) {};
		\node (b2) at (-3,2){};
		\node (b3) at (-5,2){};
		\node (b4) at (-7,2) {};

		\node (c1) at (-7,0) {};
		\node (c2) at (7,0) {};

		\node at (6,1) {$\dots$};		
		\node at (-6,1) {$\dots$};		
		
		\node at (0,-0.5) {$\inv(P)=(\alpha_1,\alpha_0+2\alpha_1,2\alpha_0+3\alpha_1,-2\alpha_0-3\alpha_1)$};

		\draw[thick] (0)--(a1);
		\draw[thick] (0)--(a2);
		\draw[thick] (0)--(a3);
		\draw[thick] (0)--(a4);

		\draw[thick] (0)--(b1);
		\draw[thick] (0)--(b2);
		\draw[thick] (0)--(b3);
		\draw[thick] (0)--(b4);

		\draw[dotted] (0)--(c1);
		\draw[dotted] (0)--(c2);
		
		\draw [pathcolor,densely dotted,thick,->] (0,1)  .. controls (2,1.5)  and (5,1)..  (2,1) ;

		\end{tikzpicture}}
\end{subfigure}
\caption{Geometric interpretation of inversions of a word in the Tits cone for infinite Coxeter groups.}
\label{fig:affineA1_path}
\end{figure}
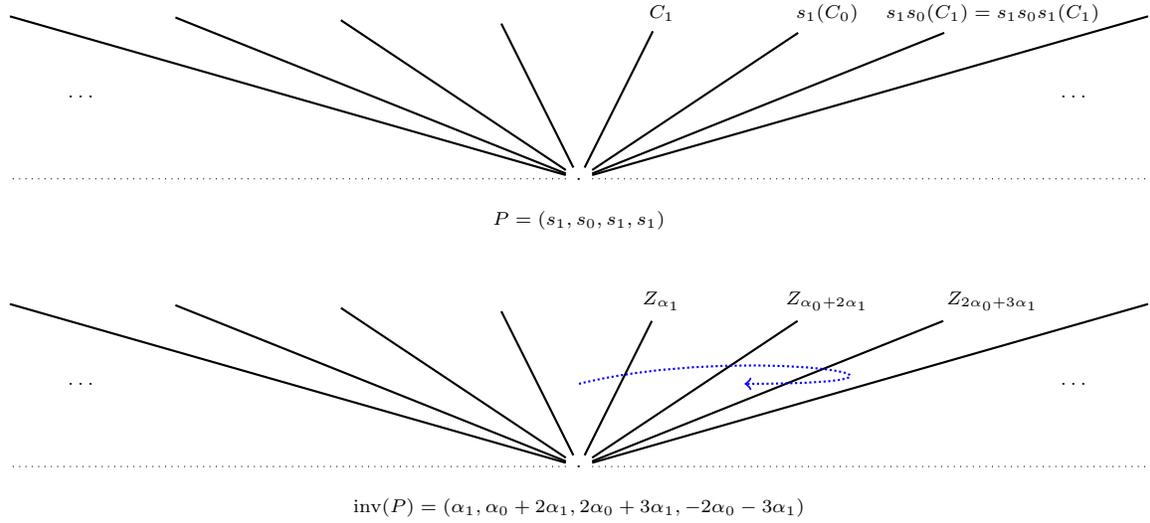

\subsection{Restriction to root subsystems}
The restriction of the list of inversions of a (non-necessarily reduced) word $P$ to a subspace~$V'\subset V$ behaves very well from a Coxeter group and root system perspective. 
The subspace $V'$ has a natural root subsystem
$
\Phi'=\Phi'^+ \sqcup  \Phi'^-
$ 
obtained by restricting~$\Phi,\Phi^+,\Phi^-$ to $V'$. We denote by $\Delta'$ and $W'$ the corresponding simple roots and Coxeter group. Indeed, the intersection of a root system with a subspace is again a root system with simple roots contained in $\Phi^+$~\cite{dyer_reflection_1990}, which is a non-trivial result for infinite Coxeter groups. 
We also consider root subsystems $\Phi'\subset \Phi$ which are not necessarily obtained as the intersection of $\Phi$ with a subspace, as happens for the root system of type $A_1\times A_1$ when viewed as a root subsystem of the root system of type~$B_2$. In this case, we also denote by $\Delta'\subset \Phi'^+$ and $W'$ the corresponding simple roots and Coxeter group.
Our main purpose is to prove the following proposition, which is the main ingredient in the \emph{Decomposition theorem of subword complexes}, Theorem~\ref{thm:docompositon}.

 \begin{proposition}\label{prop:first_inversion_in_F}
 Let $P$ be a (non-necessarily reduced) word in the generators of a (possibly infinite) Coxeter group $W$.
 The restriction of $\inv(P)$ to a root subsystem $\Phi'\subset \Phi$ is the list of inversions $\inv(P')$ of a word $P'$ in the generators of $W'$. In particular, the first root of $\inv(P)$ that belongs to $\Phi'$ is a simple root of $\Phi'$.
 \end{proposition}
 
 Again, in order to keep the intuition from finite reflection groups, we distinguish the two cases of finite and infinite Coxeter groups. 
 
\begin{proof}[Proof for finite Coxeter groups]
Let~$P=(p_1,\dots,p_r)$ and~$\inv(P)=\{\gamma_1,\dots,\gamma_r\}$ be the corresponding list of inversions, and assume $W$ is finite.
Consider the path corresponding to the word $P$ in the Coxeter complex, and its orthogonal projection to the subspace~$V'$ spanned by $\Phi'$. 
This projection starts at the fundamental chamber defined by $\Phi'$, and crosses the walls $H_{\gamma_i}$ corresponding to the inversions~$\gamma_i \in \Phi'$ in the order they appear in the list of inversions of $P$. Define $P'$ as the word in the generators of $W'$ corresponding to this path. The restriction of $\inv(P)$ to $\Phi'$ is then exactly equal to $\inv(P')$. In particular, since each wall of the fundamental chamber of $\Phi'$ is orthogonal to a simple root in~$\Delta'$, the first inversion in $\inv(P)$ that belongs to~$\Phi'$ is a simple root of $\Phi'$.
\end{proof}

\begin{proof}[Proof for infinite Coxeter groups]
The proof in the infinite case works precisely the same by changing the Coxeter complex to the Tits cone, and the orthogonal projection to the projection to the quotient space $V^*/Z_{\Phi'}$, where $Z_{\Phi'}=\{f\in V^*: \langle f, \alpha \rangle = 0 \text{ for all } \alpha\in \Phi'\}$. The projection of~$Z_\alpha$, for~$\alpha\in \Phi'$ determines the Tits cone of $\Phi'$ in this quotient space, and the projection of the path of~$P$ corresponds to a word $P'$ in the generators of $W'$.
As before, 
the restriction of $\inv(P)$ to $\Phi'$ is equal to $\inv(P')$, and the first inversion in $\inv(P)$ that belongs to~$\Phi'$ is a simple root of $\Phi'$.
\end{proof}

\section{Chipping gems out of rocks}\label{Appendix:gems}
Whenever we have a combinatorial Hopf algebra and a combinatorial basis, we can consider a Markov process on the objects indexing the basis.
We refer the reader to \cite{DiaconisPangRam, Pang2014} for the general details on this idea.
As described in \cite{Pang2014} one needs a basis of the Hopf algebra with elements that are not primitives in dimension $\ne 1$
and finite dimensional invariant subspaces of the operators
  $$m_\alpha \circ \Delta_\alpha. $$
On the combinatorial Hopf algebra of permutations, the operator $m\circ\Delta$ correspond to the riffle-shuffle and the operator $m_{1,n-1}\circ\Delta_{1,n-1}$ correspond
to the top-to-random shuffle. Here we describe informally what the operation $m_{1,n-1}\circ\Delta_{1,n-1}$ is on the Hopf algebra of subword complexes $\YY$.
We remark that an element $(W,Q,\pi,I)$ of degree $n>1$ is never primitive. Moreover given $(W,Q,\pi,I)$ of degree $n$
  \begin{equation}\label{eq:breaking} m_{\alpha}\circ\Delta_{\alpha}(W,Q,\pi,I) = \sum c_{(W',Q',\pi',I')} (W',Q',\pi',I'), \end{equation}
where $(W',Q',\pi',I')$ are of degree $n$, the length of the word $\ell(Q')\le \ell(Q)$ and for any generator $s_{i'},s_{j'}$ of $W'$, we have that $m_{i'j'}=2$ or is equal $m_{ij}$ for some generTOR $S_I,S_J$ OF $W$.

 We thus have that  $m_\alpha \circ \Delta_\alpha\colon k[Y_n^{m,\ell}]\to k[Y_n^{m,\ell}]$ and 
 $k[Y_n^{m,\ell}]$ is finite dimentional.  
The theory of  \cite{Pang2014} can thus be applied of this space to get Markov processes
on the $(W,Q,\pi,I)\in Y_n^{m,\ell}$.  In Equation~\eqref{eq:breaking}, for $m_{1,n-1}\circ\Delta_{1,n-1}$, the group $W'$ must be of the form $W'=A_1\times W''$ and $Q'=(s_1,\ldots,s_1)Q''$. The number of copies of $s_1$ in $Q'$ may vary depending on the term in the expansion.

The Markov process induced by the top-to-random shuffle for  fixed $n,k$ on $\YY$ can be interpreted as follows.
We imagine that the elements  $(W,Q,\pi,I)$ are types of rocks. More precisely, we think that if $W=W_1\times W_2\times\cdots W_k$ and the $W_i$ are indecomposable, then $(W,Q,\pi,I)$
is exactly $k$ rocks of certain types.
The operator $m_{1,n-1}\circ\Delta_{1,n-1}$ can be though off as a small hammer hitting on the rocks.
The expansion in Equation~\eqref{eq:breaking} describes the types of rocks $(W',Q',\pi',I')$ we can get from one small hammer hit on the rocks $(W,Q,\pi,I)$.
The result is always a small chipped rock (of type $A_1$) and what is left (of type $W''$) of the original rocks.
The little chipped rock of type $A_1$ can be of different quality, namely the number of $s_1$'s in the word $Q'$.
We thus call this chipped rock a {\sl gem} and the quality of the gem is proportional to the number of $s_1$ in $Q'$.
Iterating this process would break any rocks into gems. That is, the stable states  of the Markov process are of the form $W=A_1\times A_1\times \cdots A_1$ where we get different  quality of gems.
This is similar to the rock breaking process of \cite{DiaconisPangRam, Pang2014}, but here an initial state of rocks may lead to more than one possible stable outcome.
We then leave to the reader the  pleasure of studying the Markov process of chipping gems out of these kind of rocks.

%
%

\bibliographystyle{amsalpha}
\bibliography{HopfAlgebra_ArXive.bib}

\end{document}